\documentclass{amsart}

\usepackage{epsf,amscd,amssymb}
 \usepackage{graphicx,amscd,amssymb,bbm}
\usepackage{graphics}
\usepackage{rotating}
\usepackage{url}

\usepackage[backref,bookmarks=true,colorlinks=true,pdfstartview=FitV,linkcolor=blue, citecolor=red,urlcolor=green]{hyperref}  

\newtheorem{theorem}{Theorem}[section]
\newtheorem{lemma}[theorem]{Lemma}
\newtheorem{corollary}[theorem]{Corollary}
\newtheorem{proposition}[theorem]{Proposition}

\theoremstyle{definition}
\newtheorem{definition}[theorem]{Definition}
\newtheorem{example}[theorem]{Example}

\theoremstyle{remark}
\newtheorem{remark}[theorem]{Remark}

\numberwithin{equation}{section}

\newcommand\bR{{\mathbb{R}}}

\newcommand\bZ{{\mathbb Z}}

\newcommand\GL{{\rm GL}}
\newcommand\SL{{\rm SL}}
\newcommand\PGL{{\rm PGL}}

\newcommand\PO{{\rm PO}}

\newcommand\Hom{{\rm Hom}}

\newcommand\dev{{\bf dev}}
\newcommand\SI{{\mathbb{S}}}

\newcommand\Bd{{\rm bd}}
\newcommand\clo{{\rm Cl}}
\newcommand\bdd{{\bf d}}

\newcommand\ra{\rightarrow}

\newcommand\emp{\emptyset}
\newcommand\eps{\epsilon}

\newcommand\Idd{{\rm I}}

\newcommand\rpn{\mathbbm{RP}^n}
\newcommand\rpnn{\mathbbm{RP}}

\newcommand\rep{\mathrm{rep}}

\newcommand\Def{\mathrm{Def}}
\newcommand\CDef{\mathrm{CDef}}
\newcommand\SDef{\mathrm{SDef}}
\newcommand\hol{\mathrm{hol}}
\newcommand\torb{\widetilde{\mathcal{O}}}
\newcommand\orb{\mathcal{O}}
\newcommand\Ag{{\mathrm{Ag}}}
\newcommand\Pgl{{\mathrm{PGL}}(n+1, \bR)}
\newcommand\SLpm{{{\mathrm{SL}}_{\pm}(n+1, \bR)}}

\newcommand\cR{{\mathcal{R}}}
\newcommand\cT{{\mathcal{T}}}
\newcommand\rrp{{\rm )}} 
\newcommand\rlp{{\rm (}}

\newcommand\bN{{\mathbbm N}}
\newcommand\bP{{\mathbbm P}}

\begin{document}

\title[The convex real projective manifolds and orbifolds with ends]{The convex real projective orbifolds with
radial or totally geodesic ends: a survey of some partial results}

\author{Suhyoung Choi} 
\address{Department of Mathematical Sciences \\ KAIST \\
Daejeon 305-701, South Korea
         }
         
\email{schoi@math.kaist.ac.kr}      
\thanks{This work was supported by the National Research Foundation of Korea(NRF) grant funded by the Korea government(MEST) 
(No. 2013R1A1A2056698).}


\date{\today}



\begin{abstract}
A real projective orbifold has a {\em radial end} if a neighborhood of the end is foliated by projective geodesics that develop into
geodesics ending at a common point.
It has a {\em totally geodesic end} if the end can be completed to have  the totally geodesic boundary. 

We will prove a homeomorphism 
between the deformation space of convex real projective structures on an orbifold $\orb$ with radial or totally geodesic ends 
with various conditions with the union of open subspaces of strata of  
 the subset 
\[ \Hom_{\mathcal E}(\pi_{1}(\orb), \Pgl)/\Pgl  \]
of the $\Pgl$-character variety for $\pi_{1}(\orb)$ 
given by corresponding end conditions for holonomy representations.



Lastly, we will talk about the openness and closedness of the properly (resp. strictly) convex real projective structures on a class of orbifold with
\hyperlink{term-addg}{generalized admissible ends}.

\end{abstract}

\maketitle

\tableofcontents
\listoffigures







\section{Introduction} 

\subsection{Preliminary} 
Let $H$ be a closed upper-half space $\{x \in \bR^{n}| x_{n}\geq 0\}$ with 
boundary $\partial H =\{x \in \bR^{n}| x_{n} = 0\}$. 
An \hypertarget{term-orb}{{\em orbifold}} $\orb$ is a second countable Hausdorff space 
where each point $x$ has a neighborhood with a chart $(U, G, \phi)$ consisting of 
\begin{itemize}
\item a finite group $G$ acting on $U$ an open subset of $H$, 
\item $\phi: U \ra \phi(U)$ inducing a homeomorphism $U/G \ra \phi(U)$ to a neighborhood 
$\phi(U)$ of $x$. 
\end{itemize} 
Also, these charts are compatible in some obvious sense
as explained by Satake and Thurston.
Such a triple $(U, G, \phi)$ is called a model of a neighborhood of $\orb$. 
The \hypertarget{term-bdorb}{{\em orbifold boundary}} $\partial \orb$ is the set of points with only models of form $(U, G, \phi)$ 
where $U$ meets $\partial H$.  A \hypertarget{term-corb}{{\em closed orbifold}} is a compact orbifold with empty boundary. 
The orbifolds in this paper are  
the quotient space of a manifold under the action of a finite group. 
(See Chapter 13 of Thurston  \cite{Thnote} or more modern Moerdijk \cite{moer}.)

We will study properly convex real projective structures on such orbifolds. 
A properly convex real projective orbifold is 
the quotient $\Omega/\Gamma$ of a properly convex domain $\Omega$
in an affine space in $\mathbbm{RP}^{n}$ by a group $\Gamma$, 
$\Gamma \subset \Pgl$, of projective automorphisms acting on $\Omega$
properly but maybe not freely. 
Let $\pi_{1}(\orb)$ denote the orbifold-fundamental group of $\orb= \Omega/\Gamma$, 
which is isomorphic to $\Gamma$. 
Given an orbifold $\orb$, a  {\em properly convex real projective structure} on $\orb$
is a diffeomorphism $f:\orb \ra \Omega/\Gamma$ for a properly convex 
real projective orbifold of form $\Omega/\Gamma$ as above. 
Finite volume complete hyperbolic manifolds are examples since we can use 
the Klein model and identify $\Omega$ as the model space
where $\Gamma$ is the hyperbolic isometry group, which acts projectively on $\Omega$. 
(See Example \ref{exmp:hyp}.) 

For closed $n$-dimensional orbifolds, these structures are somewhat well studied by Benoist \cite{Ben0}
generalizing the previous work for surfaces by Goldman \cite{Gconv} and Choi-Goldman \cite{CG}. 
The work of Cooper, Long, and Thistlethwaite \cite{CLT1} and \cite{CLT2}
showed the existence of deformations for some hyperbolic $3$-orbifolds that are explicitly computable. 
Currently, there seems to be more interest in this field due to these and other developments. 
(For a recent survey, see \cite{CLM}.)

A {\em strongly tame} orbifold is an orbifold that has a compact suborbifold whose complement is 
homeomorphic to a disjoint union of closed $(n-1)$-orbifolds times intervals. 
The theory is mostly applicable to strongly tame orbifolds that are not manifolds, and
is most natural in this setting. 
In fact, the theory is mostly adopted for Coxeter orbifolds, i.e., orbifolds based on
convex polytopes with faces silvered, and also, orbifolds that are doubles of these.
(See Section \ref{sub:rem}.)

One central example to keep in mind is the tetrahedron with silvered sides and edge orders equal to $3$.
This orbifold admits a complete hyperbolic structure. Also, it admits deformations to convex 
real projective orbifolds. The deformation space of real projective structures is homeomorphic to a four cell. 
(See Choi \cite{poly} and Marquis \cite{ludo}.) We can also take the double of this orbifold.
The deformation space is $5$-dimensional 
and can be explicitly computed in \cite{schoimath}. 
(See Chapter 7 of \cite{convMa}.)
Except for ones based on tetrahedra, 
complete hyperbolic Coxeter $3$-orbifolds with all edge orders $3$ have at least 
six dimensional deformation spaces by Theorem 1 of  Choi-Hodgson-Lee \cite{CHL}. 

Another well-known prior example is due to Tillmann: This is a complete hyperbolic orbifold 
on a complement of two-points $p, q$ in the $3$-sphere where the singularities are 
two simple arcs connecting $p$ to itself and $q$ to itself forming a link of index $1$
and another simple arc connecting $p$ and $q$.  These arcs have $\bZ_{3}$ as the local group. 
Heard, Hodgson, Martelli, and Petronio \cite{heard} labelled this orbifold $2h\underbar{\,\,}1\underbar{\,\,}1$. 
The dimension is computed to be $2$  
by Porti and Tillman \cite{PTp}. 
(See Figure \ref{fig:exmp1}.)

 For all these examples, we know that some horospherical ends deform to lens-type radial ones and vice versa. 
 We can also obtain totally geodesic ends by ``cutting off'' some radial ends. 
 We call the phenomenon ``cusp opening''. 
Benoist \cite{Ben4} first found such phenomena for a Coxeter $3$-orbifold. 
We also had the many numerical and theoretical 
results for above Coxeter $3$-orbifolds which we plan to write more explicitly in a later paper \cite{CGLM}. 
Also, Greene  \cite{Greene} found many such examples using explicit computations. 
Recently, Ballas, Danciger, and Lee \cite{BDL}  found these phenomena using cohomology arguments for complete
finite-volume hyperbolic $3$-manifolds as explained in their MSRI talks in 2015. 
Some of the computations are available from the authors. 


\begin{figure}[h]

\centerline{\includegraphics[width=8cm,trim={0 {1cm} 0 {1cm}},clip]{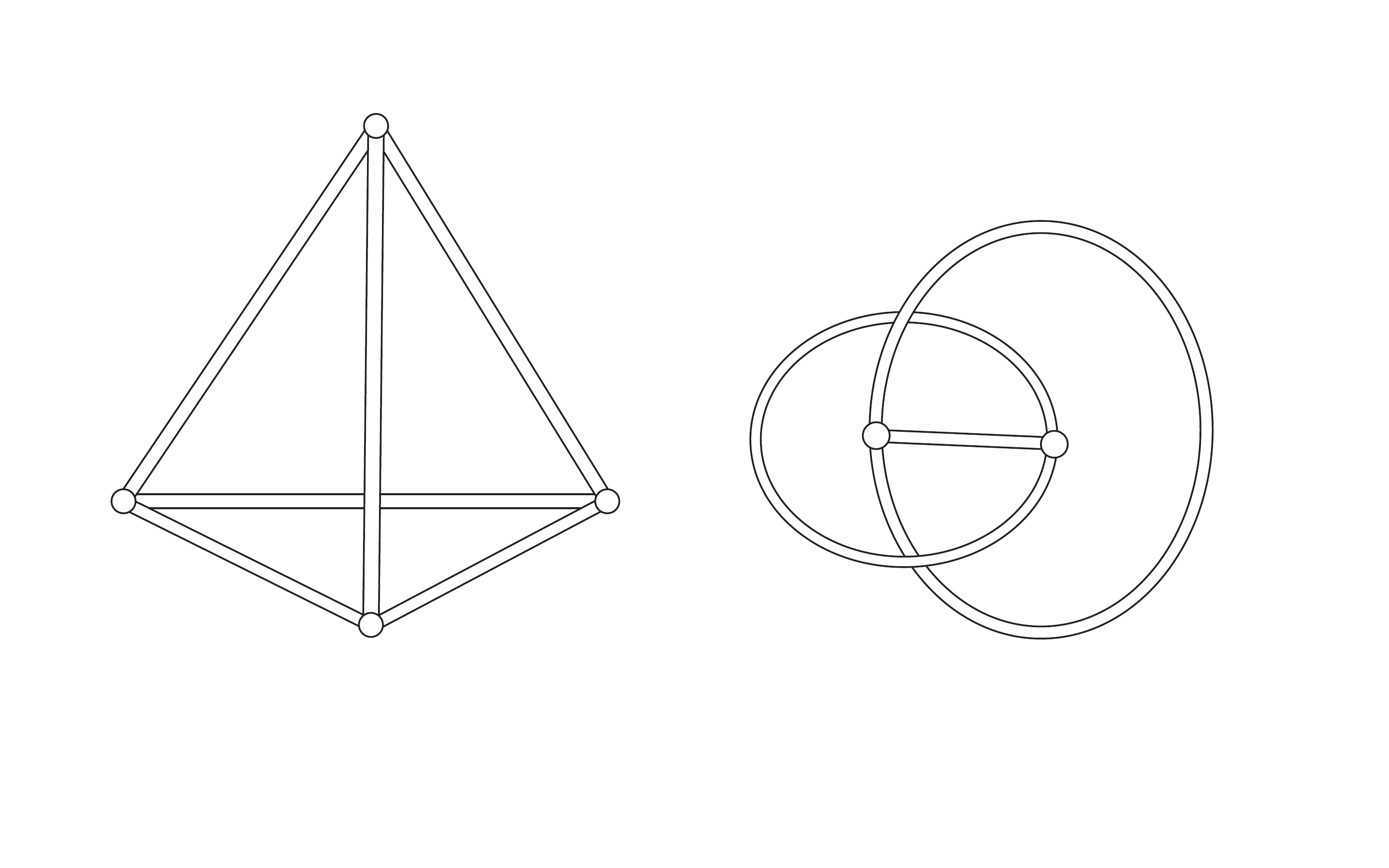}}

\caption{The singularities of the our orbifolds, the double of a tetrahedral reflection orbifold with orders $3$ and 
Tillmann's orbifold in the $3$-spheres. The white dots indicate the points removed. The edges are all of order $3$.}
\label{fig:exmp1}
\end{figure}



\subsection{Main results} \label{sub:mainresults}
We concentrate on studying the ends that are well-behaved, i.e., 
ones that are foliated by lines or are totally geodesic. 
In this setting we wish to study the deformation spaces of the convex real projective 
structures on orbifolds with some boundary conditions 
using the character varieties. Our main aim is 
\begin{itemize}
\item to identify
the deformation space of 
convex real projective structures on an orbifold $\orb$ with certain boundary conditions
with the union of some open subsets of strata of $\Hom(\pi_{1}(\orb), \PGL(n+1, \bR))/\PGL(n+1,\bR)$ 
defined by conditions corresponding to the boundary conditions. 
\end{itemize} 
This is an example of the so-called Ehresmann-Thurston-Weil principle \cite{Weil}.
(See Canary-Epstein-Green \cite{Canary}, Goldman \cite{Goldman3}, Lok \cite{Lok}, 
Bergeron-Gelander \cite{BG04}, and Choi \cite{dgorb}.)
 The precise statements are given in \hyperlink{thm-BC}{Theorems \ref{thm-B} and \ref{thm-C}}.

See Definition \ref{defn:IE} for the condition (IE) and (NA). We use the notion of {\em strict convexity with respect to ends}
as defined in Definition \ref{defn:strict}.  
Our main result is the following as a corollary of Theorem \ref{thm-closed1}: 

\begin{corollary}\label{cor-closed1} 
Let $\mathcal{O}$ be a noncompact strongly tame \hyperlink{term-sspc}{SPC $n$-orbifold} 
with \hyperlink{term-addg}{generalized admissible ends}
and satisfies {\rm (IE)} and {\rm (NA)}. 
Assume $\partial \orb =\emp$, and 
that the nilpotent normal subgroups of every finite-index subgroup of $\pi_1(\mathcal{O})$ are trivial. 
Then 
\hyperlink{term-hol}{$\hol$} 
maps the deformation space $\CDef_{{\mathcal E}, \mathrm{u, ce}}(\mathcal{O})$ of 
\hyperlink{term-spc}{SPC-structures} on $\mathcal{O}$ 
homeomorphically to a union of components of 
\[\rep_{{\mathcal E}, \mathrm{u, ce}}^{s}(\pi_1(\mathcal{O}), \PGL(n+1, \bR)).\]  
The same can be said for $\SDef_{{\mathcal E}, \mathrm{u, ce}}(\mathcal{O})$.
\end{corollary}

These terms will be defined more precisely later on in Sections \ref{subsub-char} and \ref{subsub-maintheorems}. 
Roughly speaking,   $\CDef_{{\mathcal E}}(\mathcal{O})$ (resp. $\SDef_{{\mathcal E}}(\mathcal{O})$)
is the deformation spaces of properly convex (resp. strictly properly convex)
real projective structures with conditions on ends that each end holonomy group fixes a point. 
 \[\rep_{{\mathcal E}}^s(\pi_1(\mathcal{O}), \PGL(n+1, \bR))\] is the space of characters each of whose end 
 holonomy group fixes a point.
\[\CDef_{{\mathcal E}, \mathrm{u, ce}}(\mathcal{O}) \hbox{ (resp. } \SDef_{{\mathcal E}, \mathrm{u, ce}}(\mathcal{O})) \] is 
the deformation space of properly convex (resp. strictly properly convex)
real projective structures with conditions on ends that each end has a lens-cone neighborhood or a horospherical one, and 
each end holonomy group fixes a unique point. 
 \[\rep_{{\mathcal E}, \mathrm{u, ce}}^s(\pi_1(\mathcal{O}), \PGL(n+1, \bR))\] is the space of characters each of whose end 
 holonomy group fixes a unique point and acts on a lens-cone or a horosphere. 

Our main examples satisfy this condition: 
 Suppose that a strongly tame properly convex $3$-orbifold $\orb$ with radial ends admits a finite volume 
 complete hyperbolic structure and has radial ends only and any end neighborhood contains singularities of dimension $1$
 of order $3$ or $6$, and end orbifolds have base spaces homeomorphic to disks or spheres. 
The theory simplifies by Corollary \ref{cor:hyperbolic},  i.e., each end is always of lens-type or horospherical,
so that 
\[  
\SDef_{{\mathcal E}, \mathrm{u, ce}}(\mathcal{O})= \SDef_{{\mathcal E}}(\mathcal{O}). \]
 Corollary \ref{cor-closed2} applies to these cases, and the space under $\hol$ maps homeomorphically 
to a union of components of  
\[\rep_{{\mathcal E}}^{s}(\pi_1(\mathcal{O}), \PGL(4, \bR)).\]


For a strongly tame Coxeter orbifold $\orb$ of dimension $n \geq 3$ admitting  a complete hyperbolic structure, 
$\hol$ is a homeomorphism from 
\[ \SDef_{{\mathcal E}, \mathrm{u, ce}}(\mathcal{O})  = \SDef_{{\mathcal E}}(\mathcal{O})\]
to a union of components of  
\[\rep_{{\mathcal E}}^{s}(\pi_1(\mathcal{O}), \PGL(n+1, \bR))\]
by Corollary \ref{cor-closed3}. 
For this theory, we can consider a Coxeter orbifold based on a convex polytope admitting a complete hyperbolic structure
 with all edge orders equal to $3$.  More specifically, we can consider 
 a hyperbolic ideal simplex or a hyperbolic ideal cube with such structures. 
 (See Choi-Hodgson-Lee \cite{CHL} for examples of $6$-dimensional deformations.)

 
 One question is whether we can remove the stability condition on the target character varieties. (For closed orbifolds, see Theorem 4.1 of \cite{CLM}
 essentially following from Benoist \cite{Ben3}.). 
 We plan to prove this for many interesting orbifolds, such as strongly tame orbifolds.

\subsection{Remarks} \label{sub:rem} 
We give some remarks on our results here:
The theory here is by no means exhaustive final words. 
We have a somewhat complicated theory of ends \cite{endclass}, \cite{End1}, \cite{End2}, \cite{End3},
which is used in this paper. Instead of publishing some of these and \cite{convMa}. 
We will try to refine and generalize the theory and put into a monograph \cite{newbook} in a near future. 
Our boundary condition is very restrictive  in some sense. With this it could be said that the above theory is not
so surprising.  

Ballas, Cooper, Leitner, Long,  and Tillman  have different restrictions on ends
and they are working with manifolds. 
The associated end neighborhoods have nilpotent holonomy groups. (See \cite{CLT2}, \cite{CLT3}, \cite{Leitner1}, \cite{Leitner2}, 
\cite{Leitner3}, 
and \cite{BCL}).
They are currently developing the theory of ends and the deformation theory based on this assumption.
Of course, we expect to benefit and thrive from many interactions between the theories
as they happen in multitudes of fields. 

Originally, we developed the theory for orbifolds as given in papers of Choi \cite{poly}, Choi, Hodgson, and Lee \cite{CHL}, and \cite{convMa}. 
However, the recent examples of Ballas \cite{Ballas}, \cite{Ballas2}, and Ballas, Danciger, and Lee \cite{BDL} 
can be covered using \hyperlink{term-fixings}{fixing sections}. 
Also, differently from the above work, we can allow ends with hyperbolic holonomy groups. 

\subsection{Outline} 

In Section \ref{sec-defn}, we will go over real projective structures.  
We will define end structures of convex real projective orbifolds. 
We first discuss totally geodesic ends and define \hyperlink{term-lensc}{lens condition} for these totally geodesic ends. 
Then we define radial ends and radial foliation marking for radial ends.
We define end orbifolds, and horospherical ends. 
We define the space of characters  and the deformation spaces of convex real projective structures on orbifolds with 
radial or totally geodesic ends. Finally, we discuss the local homeomorphisms between the subsets of deformation spaces and those of character varieties. 
Here, we are not yet concerned with convexity. 


In Section \ref{sec-convr}, we will discuss the known facts about the convex real projective orbifolds 
including the Vinberg duality result. 

In Section \ref{sec-endt}, we will discuss the end theory. We introduce pseudo-ends and pseudo-end neighborhoods, 
and pseudo-end fundamental groups. We introduce admissible groups and admissible ends. 
We introduce the lens-conditions for radial ends and totally geodesic ends. 

In Section \ref{sec-relhyp}, we will relate the relative hyperbolicity of the fundamental groups of 
strongly tame properly convex real projective orbifolds with the ``relative'' strict convexity of the real projective structures. 
(See Section \ref{sub-Bowditch}.) 
In Section \ref{sub-spc}, we define the stable  properly convex real projective orbifolds relative to ends. 
We will show that under mild conditions properly convex strongly tame real projective orbifolds with \hyperlink{term-addg}{generalized admissible ends}
have \hyperlink{term-st}{stable} holonomy groups.

In Section \ref{sec-clopen}, we will state our main results. 
In Section \ref{sub-open}, we will state that a holonomy homomorphism map
from a deformation space with end conditions to the character variety is a homeomorphism 
to a union of open subsets of strata of the character variety. This map is injective. 
In Section \ref{sub-closed}, we will say about the closedness.  

\subsection{Acknowledgement} 

We thank many discussions with Samuel Ballas, Daryl Cooper, Jeffrey Danciger, William Goldman,  Craig Hodgson,  Gye-Seon Lee, Darren Long, 
Joan Porti, and Stephan Tillmann. This research are supported and inspired by many difficult computer computations done by them.
We benefited from the helpful discussions with Stephan Tillman in May of 2008, and ones with
Cooper at ICERM, Brown University, in September 2013, and the ones with
Samuel Ballas, Daryl Cooper, and Darren Long at the UC Santa Barbara in June of 2015.  
We appreciate their generosity as well as their deep insights and clarifications.
We also thank the hospitality of MSRI during Spring of 2015 where some parts of this research was carried out. 

\section{Preliminary} \label{sec-defn}

\subsection{Basic definitions} 

\subsubsection{Topological notation} 
Define $\Bd A$ for a subset $A$ of $\rpn$, $n \geq 2$, (resp. in $\SI^n$) to be the {\em topological boundary} in $\rpn$ (resp. in $\SI^n$) \index{boundary!topological}
and define $\partial A$ for a manifold or orbifold $A$ to be the {\em manifold or orbifold boundary} \index{boundary!manifold}
and $A^o$ denote the manifold interior. \index{interior!manifold}
The closure $\clo(A)$ of a subset $A$ of $\rpn$ (resp. of $\SI^n$) is the topological closure in $\rpn$ (resp. in $\SI^n$). \index{$\clo(\cdot)$}

\subsubsection{The Hausdorff metric}
Recall the standard elliptic metric $\bdd$ on $\rpn$ (resp. in $\SI^n$).  \index{Hausdorff metric} \index{elliptic metric} \index{$\bdd$}
Given two sets $A$ and $B$ of $\rpn$ (resp of $\SI^n$), 
\[\bdd(A, B):=\inf \{\bdd(x, y)| x \in A, y \in B\}.\]
We can let $A$ or $B$ be points as well obviously. 


The \hypertarget{term-hmetric}{{\em Hausdorff distance}} 
between two convex subsets $K_1, K_2$ of $\rpn$ (resp. of $\SI^n$) is defined by 
\[ \bdd^H(K_1, K_2) = \inf\big\{\eps\geq 0\, \big| \,  \clo(K_1) \subset N_\eps(\clo(K_2)), \clo(K_2) \subset N_\eps(\clo(K_1))\big\}\]
where $N_\eps(A)$ is the $\eps$-$\bdd$-neighborhood of $A$ under the standard metric $\bdd$ of $\rpn$ (resp. of $\SI^n$) for $\eps >0$. 
$\bdd^H$ gives a compact Hausdorff topology on the set of all compact subsets of $\rpn$ (resp. of $\SI^n$).
(See p. 281 of \cite{Munkres}.)

We say that a sequence of sets $\{K_i\}$ \hypertarget{term-geoc}{{\em geometrically converges}} 
to a set $K$ if $\bdd^H(K_i, K) \ra 0$. \index{$\bdd^H$}
If $K$ is assumed to be closed, then the geometric limit is unique.

\begin{lemma} \label{lem:bdconv}
Suppose that a sequence $\{K_i\}$  of compact convex domains geometrically converges to 
a compact convex domain $K$ in $\rpn$ {\rm (}resp. in $\SI^n$\,{\rm).} 
Then
\begin{equation}\label{eqn:partialKi}
 \bdd^H(\partial K_i, \partial K) \ra 0.
 \end{equation}   
 \end{lemma} 


\subsubsection{Real projective structures} \label{sub:prel}

Given a vector space $V$, we let $\bP(V)$ denote the space obtained by taking the quotient space of \index{$\bP(\cdot)$}
$V-\{O\}$ under the equivalence relation
\[v\sim w \hbox{ for } v, w \in \bR^{n+1} -\{O\}  \hbox{ iff } v = s w, \hbox{ for } s \in \bR -\{0\}.\] 
We let $[v]$ denote the equivalence class of $v \in V -\{O\}$. 
Recall that the projective linear group $\PGL(n+1, \bR)$ acts on $\rpn$, i.e.,
$\bP(\bR^{n+1})$, in a standard manner. 

Let $\mathcal{O}$ be a noncompact strongly tame $n$-orbifold where the orbifold boundary is not necessarily empty. 
\begin{itemize} 
\item A {\em real projective orbifold} is an orbifold with a geometric structure modelled on $(\rpn, \PGL(n+1, \bR))$. \index{real projective orbifold} 
(See Thurston \cite{Thbook}, \cite{dgorb} and Chapter 6 of \cite{msj}.)
\item $\orb$ has a universal cover $\tilde{\mathcal{O}}$ where the deck transformation group $\pi_1({\mathcal{O}})$ acts on. 
\item The underlying space of $\mathcal{O}$ is homeomorphic to the quotient space $\tilde{\mathcal{O}}/\pi_1({\mathcal{O}})$.
\item A real projective structure on $\mathcal{O}$ gives us a so-called development pair  \index{real projective structure} 
$(\dev, h)$ where 
\begin{itemize} 
\item $\dev:\tilde{\mathcal{O}} \ra \rpn$ is an immersion, called the {\em developing map}, 
\item and $h:\pi_1(\mathcal{O}) \ra \PGL(n+1, \bR)$ is a homomorphism, called a {\em holonomy homomorphism}, 
satisfying 
\[\dev\circ \gamma= h(\gamma)\circ \dev \hbox{ for } \gamma \in \pi_1(\mathcal{O}).\] 
\end{itemize} 
\end{itemize} 
Let $\bR^{n+1 \ast}$ denote the dual of $\bR^{n+1}$.  \index{$\bR^{n+1 \ast}$}
Let $\rpnn^{n\ast}$ denote the dual projective space $\bP(\bR^{n+1 \ast})$.  \index{$\rpnn^{n \ast}$}
$\PGL(n+1, \bR)$ acts on $\rpnn^{n\ast}$ by taking the inverse of the dual transformation. 
Then a representation $h:\pi_{1}(\orb) \ra \PGL(n+1, \bR)$ has a dual representation 
$h^*: \pi_{1}(\orb) \ra \PGL(n+1, \bR) $ sending 
elements of $\pi_1(\orb)$ to the inverse of the dual transformation of $\bR^{n+1 \ast}$.


The complement of a codimension-one subspace of $\rpn$ can be identified with an affine space 
$\bR^n$ where the geodesics are preserved. 
We call the complement an \hypertarget{term-affs}{{\em affine subspace}}. 
A {\em convex domain} in $\rpn$ is a convex subset of an affine subspace.  \index{convex domain}
A {\em properly convex domain} in $\rpn$ is a convex domain contained in a precompact subset of an affine subspace.  \index{convex domain! properly convex} 

 
A {\em convex real projective orbifold} is a real projective orbifold 
projectively diffeomorphic to
the quotient $\Omega/\Gamma$ where $\Omega$ is a convex domain  in an affine subspace of $\rpn$
and $\Gamma$ is a discrete group of projective automorphisms of $\Omega$ acting properly. 
If an open orbifold has a convex real projective structure, it is covered by a convex domain $\Omega$ in $\rpn$. 
Equivalently, this means that the image of the developing map $\dev(\tilde{\mathcal{O}})$ for the universal cover $\tilde{\mathcal{O}}$ of 
$\mathcal{O}$ is a convex domain for the developing map $\dev$ with associated holonomy homomorphism $h$.  
Here we may assume $\dev(\tilde{\mathcal{O}})=\Omega$, and 
$\mathcal{O}$ is projectively diffeomorphic 
to $\dev(\tilde{\mathcal{O}})/h(\pi_1(\mathcal{O}))$. In our discussions, since $\dev$ is an imbedding
and so is $h$, 
$\tilde{\mathcal{O}}$ will be regarded as an open domain in $\rpn$ and 
$\pi_{1}(\mathcal{O})$ as a subgroup of $\PGL(n+1, \bR)$ in such cases. 

\begin{remark}\label{rem:SL}
Given a vector space $V$, we denote by $\SI(V)$ the quotient space of
\[ (V - \{O\})/\sim \hbox{ where } v \sim w \hbox{ iff } v = s w \hbox{ for } s > 0. \]
We will represent each element of $\PGL(n+1, \bR)$ by a matrix of determinant $\pm 1$; 
i.e., $\PGL(n+1, \bR) = \SLpm/\langle\pm \Idd \rangle$. 
Recall the covering map $\SI^{n}= \SI(\bR^{n+1}) \ra \rpn$. 
For each $g \in \PGL(n+1, \bR)$, 
there is a unique lift in $\SLpm$ preserving each component of the inverse image of $\dev(\torb)$
under $\SI^{n} \ra \rpn$.  We will use this representative. 
\end{remark}

\subsubsection{End structures} 

A {\em strongly tame} $n$-orbifold is one where the complement of a compact set is diffeomorphic 
to a union of $(n-1)$-dimensional orbifolds times intervals. 
Of course it can be compact. 

Let $\orb$ be a strongly tame $n$-orbifold.
Each end has a neighborhood diffeomorphic to a closed $(n-1)$-orbifold times an interval. 
The {\em fundamental group} of an end is the fundamental group of such an end neighborhood. 
It is independent of the choice of the end neighborhood by Proposition \ref{prop:endf}. 
\index{end orbifold} \index{end fundamental group}

An end $E$ of a real projective orbifold $\orb$ is \hypertarget{term-Tend}{{\em totally geodesic} or {\em of type T}} 
if  the following hold: \index{end!totally geodesic} \index{end!type-T}
\begin{itemize}
\item The end has an end neighborhood homeomorphic to a closed connected $(n-1)$-dimensional orbifold $B$ times a half-open interval $(0, 1]$.
\item  $B$ completes to a compact orbifold $U$ diffeomorphic to $B \times [0, 1]$ in an ambient real projective 
orbifold.
\item The subset of $U$ corresponding to $B \times \{0\}$ is the added boundary component.
\item Each point of the added boundary component has a neighborhood projectively diffeomorphic to 
the quotient orbifold of an open set $V$ in an affine half-space $P$ so that $V \cap \partial P \ne \emp$
by a projective action of a finite group. 
\end{itemize} 
The completion is called a {\em compactified end neighborhood} of the end $E$. 
The boundary component is called the \hypertarget{term-idbd}{{\em ideal boundary}} 
{\em component} of \index{boundary!ideal}
the end. Such ideal boundary components may not be uniquely determined
as there are two projectively nonequivalent ways to add boundary components of 
elementary annuli (see Section 1.4 of \cite{cdcr2}). \index{elementary annulus} 
Two compactified end neighborhoods of an end are {\em equivalent} if they contain a common compactified end neighborhood.

We also define as follows: 
\begin{itemize}
\item The equivalence class of the chosen compactified end neighborhood is called a \hypertarget{term-marking}{{\em marking}} 
of the totally geodesic end. 
\item We will also call the ideal boundary the \hypertarget{term-endo}{{\em end orbifold}} of the end.
\end{itemize} 

\begin{definition}\label{defn-lens}
A \hypertarget{term-lens}{{\em lens}} is a properly convex domain $L$ so that $\partial L$ is a union of two smooth strictly convex open disks. \index{lens}
A properly convex domain
$K$ is a \hypertarget{term-glens}{{\em generalized lens}} if $\partial K$ is a union of two open disks one of which is strictly convex and smooth
and the other is allowed to be just a topological disk.
A {\em lens-orbifold}  is a compact quotient orbifold of a lens by a properly discontinuous action of a projective group $\Gamma$.  \index{lens!orbifold} 
Thus, for two boundary components $A$ and $B$, $A/\Gamma$ and $B/\Gamma$ are homotopy equivalent to $L/\Gamma$
by the obvious inclusion maps.
\end{definition}

\begin{description} 
\item[(Lens condition)] \hypertarget{term-lensc}{The ideal boundary} is realized as a totally geodesic suborbifold in the interior of a lens-orbifold 
in some ambient real projective orbifold of $\orb$. 
\index{end!lens condition} 
\end{description}
If the lens condition is satisfied for an T-end, 
we will call it the \hypertarget{term-lensT}{{\em T-end of lens-type}}. 

Let $\torb$ denote the universal cover of $\orb$ with the developing map $\dev$. 
An end  $E$ of a real projective orbifold is \hypertarget{term-Rend}{{\em radial}} or {\em of type R} if 
the following hold: 
\begin{itemize}
\item The end has an end neighborhood $U$ foliated by properly imbedded projective geodesics. \index{end!radial} \index{end!type-R}
\item Choose a map $f: \bR \times [0, 1] \ra \orb$ so that $f|\bR\times \{t\}$ for each $t$ is a geodesic leaf of 
such a foliation of $U$.
Then $f$ lifts to $\tilde f:\bR \times [0,1] \ra \torb$ where $\dev \circ \tilde f| \bR \times \{t\}$ for each $t, t \in [0, 1]$, 
maps to a geodesic in $\rpn$ ending at a point of concurrency common for $t$. 
\end{itemize} 
The foliation  is called a {\em radial foliation} and leaves {\em radial lines} of $E$. 
Two such radial foliations $\mathcal F_{1}$ and $\mathcal F_{2}$ 
of radial end neighborhoods of an end are {\em equivalent} if the restrictions of $\mathcal{F}_{1}$ and $\mathcal{F}_{2}$ 
in an end neighborhood agree. 
A {\em radial foliation marking} is an equivalence class of radial foliations. \index{end!radial foliation marking}
The marking will give us an ideal boundary component on the end. 

\begin{definition}
A {\em real projective orbifold with radial or totally geodesic ends} is a strongly tame 
orbifold with a real projective structure where each end is an R-end or a T-end with a marking given for each.
\end{definition}

Let $\mathbbm{RP}^{n-1}_{x}$ denote the space of concurrent lines to a point $x$
where $\mathbbm{RP}^{n-1}_{x}$ is projectively diffeomorphic to $\rpnn^{n-1}$. 
The real projective transformations fixing $x$ induce real projective transformations of $\rpnn^{n-1}_{x}$. 
Radial lines in an R-end neighborhood are {\em equivalent} if they agree outside 
a compact subset. 
The space of equivalent classes of 
radial lines in an R-end neighborhood is an $(n-1)$-orbifold by the properness of the radial lines. 
The {\em end orbifold} associated with an R-end is defined as the equivalence space of radial lines in $\orb$. 
The equivalence space of radial lines in an R-end has the local structure of $\rpnn^{n-1}$ since we can lift a local 
neighborhood to $\torb$, and these radial lines lift to lines developing into concurrent lines.  
The end orbifold has a unique induced real projective structure of one dimension lower. 



\begin{example}\label{exmp:hyp}
Let $\bR^{n+1}$ have standard coordinates $x_0, x_1, \dots, x_n$, and 
let $B$ be the subset in $\rpnn^n$ corresponding to the cone given by 
\[x_0 > \sqrt{x_1^2 + \cdots + x_n^2}.\] 
The Klein model gives a hyperbolic space as $B \subset \rpnn^n$ with 
the isometry group $\PO(1, n)$, a subgroup of $\Pgl$ acting on $B$. 
Thus, a complete hyperbolic orbifold is projectively diffeomorphic to 
a real projective orbifold of $B/\Gamma$ for 
$\Gamma$ in $\PO(1, n)$. 
The interior of a finite-volume hyperbolic $n$-orbifold with rank $(n-1)$ horospherical ends and
totally geodesic boundary forms an example of a properly convex strongly tame real projective orbifold 
with radial or totally geodesic ends. 
For horospherical ends, the end orbifolds have Euclidean structures. 
(Also, we could allow hyperideal ends by attaching radial ends. Section 3.1.1 in \cite{endclass}.)
\end{example} 

\begin{example} \label{exmp-endv}
For examples, if the end orbifold of an R-end $E$ is a $2$-orbifold based on a sphere with
three singularities of order $3$, then a line of singularity is a leaf of a radial foliation. 
End orbifolds of Tillman's orbifold and the the double of a tetrahedral reflection orbifold
are examples. 
A double orbifold of a cube with edges having orders $3$ only has eight such end orbifolds.
(See Proposition 4.6 of \cite{End1} and their deformations are computed in \cite{CHL}.
Also, see Ryan Greene \cite{Greene} for the theory.)
\end{example}

\subsubsection{Horospherical ends}

An {\em ellipsoid} in $\rpnn^n=\bP(\bR^{n+1})$ (resp. in $\SI^n=\SI(\bR^{n+1})$)  \index{ellipsoid}
is the projection $C -\{O\}$ of the null cone 
\[C:=\{x \in \bR^{n+1}| B(x, x)=0\}\] 
for a nondegenerate symmetric
bilinear form $B: \bR^{n+1} \times \bR^{n+1} \ra \bR$. Ellipsoids are always equivalent by projective automorphisms of $\rpn$.
An {\em ellipsoid ball} is the closed contractible domain in an affine subspace $A$ of 
$\rpnn^n$ (resp. $\SI^n$)  bounded by an ellipsoid contained in $A$. 
A {\em horoball} is an ellipsoid ball with a point $p$ of the boundary removed. \index{horoball}
An ellipsoid with a point $p$ on it removed is called a {\em horosphere}. The {\em vertex} of \index{horosphere}
the horosphere or the horoball is defined as $p$.

Let $U$ be a horoball with a vertex $p$ in the boundary of $B$. 
A real projective orbifold that is real projectively diffeomorphic to an orbifold
$U/\Gamma_p$ for a discrete subgroup $\Gamma_p \subset \PO(1, n)$ fixing 
a point $p \in \Bd B$ is called a {\em horoball orbifold}. \index{horoball!orbifold}
A \hypertarget{term-horo}{{\em horospherical end}} is an end with an end neighborhood that is such an orbifold. 

\subsubsection{Deformation spaces and the space of holonomy homomorphisms}  \label{subsub-defspace}
An {\em isotopy} $i: \orb \ra \orb$ is a self-diffeomorphism so that  \index{isotopy}
there exists a smooth orbifold map $J: \orb \times [0, 1] \ra \orb$,  
so that 
\[i_{t}:\orb \ra \orb \hbox{ given by } i_{t}(x) = J(x, t)\] 
are self-diffeomorphisms for $t \in [0,1]$
and $i=i_1, i_0 = \Idd_{\orb}$. 
We will extend this notion strongly. 
\begin{itemize}
\item Two real projective structures $\mu_0$ and $\mu_1$ on $\orb$ with R-ends or T-ends with end markings are {\em isotopic} 
if there is an isotopy $i$ on $\mathcal{O}$ so that $i^*(\mu_0)=\mu_1$ where $i^*(\mu_0)$ is the induced structure from $\mu_0$ by $i$
where we require for each $t$
\begin{itemize}
\item $i_{t\ast}(\mu_{0})$ has a radial end structure for each radial end, 
\item $i_{t}$ sends the radial end foliation for $\mu_0$ from an R-end neighborhood to the radial end foliation for real projective 
structure $\mu_t = i_{t\ast}(\mu_{0})$ with corresponding R-end neighborhoods, 
\item $i_t$ extends to diffeomorphisms of 
the compactifications of $\orb$ using the radial foliations 
and the totally geodesic ideal boundary components for $\mu_0$ and $\mu_t$.  
\end{itemize}
\end{itemize} 
We define 
\hypertarget{term-Def}{$\Def_{\mathcal E}(\mathcal{O})$} as the deformation space of real projective structures on $\mathcal{O}$ with end marks; more precisely, 
this is the quotient space of the real projective structures on $\mathcal{O}$ satisfying the above conditions for
ends of type R and T 
under the isotopy equivalence relations.
We put on $\mathcal{O}$ a radial foliation on each end neighborhood of type R and 
attach an ideal boundary component for each end neighborhood of type T
to obtain a new compactified orbifold $\overline{\mathcal{O}}$.
We introduce the equivalence relation based on isotopies and end neighborhood 
structures and ideal boundary components. 
We may assume that the developing maps extend to the smooth maps of the universal cover $\widehat{\mathcal{O}}$ of $\overline{\mathcal{O}}$.  
The topology of such a space is  defined by the compact open 
$C^2$-topology for the space of developing maps $\dev| \widehat{\mathcal{O}}$.
(See \cite{dgorb}, \cite{Canary} and \cite{Goldman3} for more details. )



\subsubsection{The end restrictions} 
To discuss the deformation spaces, we introduce the following notions. 
The end will be either assigned an \hypertarget{term-cRend}{{\em $\cR$-type}} or a \hypertarget{term-cTend}{{\em $\cT$-type}}. \index{$\cR$-type} \index{$\cT$-type}
\begin{itemize} 
\item An $\cR$-type end is required to be radial. 
\item A $\cT$-type end is required to have totally geodesic properly convex ideal boundary 
components of lens-type or be \hyperlink{term-horo}{horospherical}. 
\end{itemize} 
A strongly tame orbifold will always have such an assignment in this paper, 
and finite-covering maps will always respect the types. 
We will fix the types for ends of our orbifolds in consideration. 


\subsubsection{Character spaces of relevance} \label{subsub-char}

Since $\orb$ is strongly tame, the fundamental group $\pi_{1}(\orb)$ is finitely generated. 
Let $\{g_1, \dots, g_m\}$ be a set of generators of $\pi_1(\orb)$. 
As usual $\Hom(\pi_1(\mathcal{O}), G)$ for a Lie group $G$ has an {\em algebraic topology} as a subspace 
of $G^m$. This topology is given by the notion of {\em algebraic convergence}
\[\{h_i\} \ra h \hbox{ if } h_i(g_j) \ra h(g_j) \in G \hbox{ for each } j, j=1, \dots, m.\] 
A conjugacy class of representation is called a {\em character} in this paper. 

The {\em $\PGL(n+1, \bR)$-character variety} $\rep(\pi_1(\mathcal{O}), \PGL(n+1,\bR))$ is the quotient space of \index{character variety}
the homomorphism space \[\Hom(\pi_1(\mathcal{O}), \PGL(n+1,\bR))\] where $\PGL(n+1,\bR)$ acts by conjugation
\[h(\cdot) \mapsto g h(\cdot) g^{-1} \hbox{ for } g \in \PGL(n+1,\bR).\]
Similarly, we define 
\[ \rep(\pi_1(\mathcal{O}), \SLpm):= \Hom(\pi_{1}(\orb), \SLpm)/\SLpm \] as 
the $\SLpm$-character variety. 

A representation or a character is \hypertarget{term-st}{{\em stable}} if the orbit of it or its representative is closed and the stabilizer is finite under 
the conjugation action in  
\[\Hom(\pi_1(\mathcal{O}), \PGL(n+1,\bR)) \hbox{ (resp. } \Hom(\pi_{1}(\orb), \SLpm)).\]
By Theorem 1.1 of \cite{JM}, a representation $\rho$ is stable if and only if it is irreducible and 
no proper \hyperlink{term-pgroup}{parabolic subgroup} \index{character!stable} \index{representation!stable}
contains the image of $\rho$. 
The stability and the irreducibility are open conditions in the Zariski topology. 
Also, if the image of $\rho$ is Zariski dense, then $\rho$ is stable. 
$\PGL(n+1, \bR)$ acts properly on the open set of stable representations
in $\Hom(\pi_1(\mathcal{O}), \PGL(n+1,\bR))$. Similarly, 
$\SLpm$ acts so on $\Hom(\pi_{1}(\orb), \SLpm)$.
(See \cite{JM} for more details.)

A representation of a group $G$ into $\PGL(n+1, \bR)$ or $\SLpm$ is {\em strongly irreducible} \index{representation!strongly irreducible}
if the image of every finite index subgroup of $G$ 
is irreducible. 
Actually, many of the orbifolds have strongly irreducible and stable holonomy homomorphisms by Theorem \ref{thm-sSPC}.

An {\em eigen-$1$-form} of a linear transformation $\gamma$ is a linear functional $\alpha$ in $\bR^{n+1}$ 
so that \index{eigen-$1$-form}
$\alpha \circ \gamma = \lambda \alpha$ for some $\lambda \in \bR$. 
We recall the lifting of Remark \ref{rem:SL}. 

\begin{itemize} 
\item  \[\Hom_{\mathcal E}(\pi_1(\mathcal{O}), \PGL(n+1,\bR))\] to be the subspace of representations $h$ satisfying  
\begin{description}
\item[The vertex condition for $\cR$-ends] $h|\pi_{1}(\tilde E)$ 
has a nonzero common eigenvector of positive eigenvalues for 
the lift of $h(\pi_{1}(\tilde E))$ in $\SLpm$ \index{vertex condition}
for each
\hyperlink{term-cRend}{$\cR$-type p-end} fundamental group $\pi_{1}(\tilde E)$,
and
\item[The hyperplane condition for $\cT$-ends ] $h|\pi_{1}(\tilde E)$ 
acts on a hyperplane $P$ 
for each 
\hyperlink{term-cTend}{$\cT$-type p-end} fundamental group $\pi_{1}(\tilde E)$
discontinuously and cocompactly on a lens $L$, a properly convex domain with $L^o\cap P = L \cap P\ne \emp$
or a horoball tangent to $P$. 
\end{description} 
\item We denote by 
 \[\Hom^{s}(\pi_1(\mathcal{O}), \PGL(n+1,\bR))\]
 the subspace of stable and irreducible representations, and define
  \[\Hom_{\mathcal E}^{s}(\pi_1(\mathcal{O}), \PGL(n+1,\bR))\]
 to be \[\Hom_{\mathcal E}(\pi_1(\mathcal{O}), \PGL(n+1,\bR)) \cap \Hom^{s}(\pi_1(\mathcal{O}), \PGL(n+1,\bR)).\]
 
 \item We define
 \[\Hom_{\mathcal E, \mathrm{u}}(\pi_1(\mathcal{O}), \PGL(n+1,\bR))\] 
 to be the subspace of 
 representations $h$  where
 \begin{itemize}
\item  $h|\pi_{1}(\tilde E)$  has a unique common eigenspace of dimension $1$ in $\bR^{n+1}$ with positive eigenvalues
for its lift in $\SLpm$  for each p-end fundamental group $\pi_{1}(\tilde E)$ 
of $\mathcal{R}$-type 
and 
\item $h|\pi_{1}(\tilde E)$  has a common null-space $P$ of eigen-$1$-forms 
uniquely satisfying the following: 
\begin{itemize}
\item $\pi_{1}(\tilde E)$ acts properly on a lens $L$ with $L \cap P$ with nonempty interior in $P$
or 
\item $H-\{p\}$ for a horosphere $H$ tangent to $P$ at $p$
\end{itemize} 
for each p-end fundamental group $\pi_{1}(\tilde E)$ of the end of $\mathcal{T}$-type.


\end{itemize} 
\end{itemize}

\begin{remark} 
The above condition for type $\mathcal{T}$ generalizes the principal boundary condition
for real projective surfaces. 
\end{remark}


 Suppose that there are no ends of $\mathcal{T}$-type. 
Since each $\pi_1(\tilde E)$ is finitely generated and there is only finitely many conjugacy 
classes of $\pi_{1}(\tilde E)$, 
 \[\Hom_{\mathcal E}(\pi_1(\mathcal{O}), \PGL(n+1,\bR))\]
 is a closed semi-algebraic subset. 
 \[\Hom_{\mathcal E, \mathrm{u}}(\pi_1(\mathcal{O}), \PGL(n+1,\bR))\] 
 is an open subset of this closed semi-algebraic subset by Lemma \ref{lem:ucont}. 

\begin{lemma} \label{lem:ucont} 
Let $V$ be a semialgebraic subset of $\PGL(n+1, \bR)^{n}$. 
For each  $(g_{1}, \dots, g_{n}) \in V$, 
we arbitrarily choose a maximal eigenspace $E_{i}(g_{i}) \subset \bR^{n+1}$ corresponding to the eigenvalue $\lambda(g_{i})$
where  $\bigcap_{i=1}^{n} E_{i}(g_{i}) \ne \{0\}$ on every point of semialgebraic subset $V$.
We assume that for each $i$, $(g_{1}, \dots, g_{n}) \in V \mapsto E_{i}(g_{i})\subset \bR^{n+1}$ has a nonzero continuous section on $V$. 
Then the dimension function of the intersection $\bigcap_{i=1}^{n} E_{i}(g_{i})$ 
 is upper semi-continuous in $V$. 
\end{lemma} 
\begin{proof} 
Since  the limit subspace of $E_{i}(g)$ is contained in an eigenspace of $g$, 
this follows. 
\end{proof} 


When there are ends of $\mathcal{T}$-type, 
\[\Hom_{\mathcal E}(\pi_1(\mathcal{O}), \PGL(n+1,\bR)) \hbox{ and } \Hom_{\mathcal E, \mathrm{u}}(\pi_1(\mathcal{O}), \PGL(n+1,\bR))\] 
 are the unions of open subsets of this closed semi-algebraic subsets
 since we have to consider the lens condition. 

We define
\begin{itemize} 
\item  \[\rep_{\mathcal E}(\pi_1(\mathcal{O}), \PGL(n+1,\bR))\] to be
 \[\Hom_{\mathcal E}(\pi_1(\mathcal{O}), \PGL(n+1,\bR))/\PGL(n+1, \bR).\]
\item We denote by 
 \[\rep^{s}_{\mathcal E}(\pi_1(\mathcal{O}), \PGL(n+1,\bR))\]
 the subspace of \[\rep_{\mathcal E}(\pi_1(\mathcal{O}), \PGL(n+1,\bR))\]
 of \hyperlink{term-st}{stable} and irreducible characters.
 \item 
 We define
 \[\rep_{\mathcal E, \mathrm{u}}(\pi_1(\mathcal{O}), \PGL(n+1,\bR))\] to be
 \[\Hom_{\mathcal E, \mathrm{u}}(\pi_1(\mathcal{O}), \PGL(n+1,\bR))/\PGL(n+1, \bR).\]
 \item We define 
\begin{align} 
&  \rep_{{\mathcal E}, \mathrm{u}}^s(\pi_1(\mathcal{O}), \PGL(n+1,\bR)) \nonumber \\ 
 & := \rep^s(\pi_1(\mathcal{O}), \PGL(n+1,\bR)) \cap \rep_{{\mathcal E}, \mathrm{u}}(\pi_1(\mathcal{O}), \PGL(n+1,\bR)).
  \end{align} 
\end{itemize} 

Note that when there are no $\mathcal{T}$-type ends, 
 \[\rep^{s}_{\mathcal E}(\pi_1(\mathcal{O}), \PGL(n+1,\bR))\]
 is a closed subset of 
 \[\rep^{s}(\pi_1(\mathcal{O}), \PGL(n+1,\bR)), \hbox{ and }\]
 \[\rep^{s}_{{\mathcal E}, \mathrm{u}}(\pi_1(\mathcal{O}), \PGL(n+1,\bR))\]
 is an open subset of 
 \[\rep^{s}_{\mathcal E}(\pi_1(\mathcal{O}), \PGL(n+1,\bR)).\]



Note that elements of $\Def_{\mathcal E}(\orb)$ have characters in \[\rep_{{\mathcal E}}(\pi_1(\mathcal{O}), \PGL(n+1,\bR)).\]
Denote by $\Def_{{\mathcal E}, \mathrm{u}}(\mathcal{O})$ the subspace of $\Def_{{\mathcal E}}(\mathcal{O})$ of equivalence classes of real projective structures 
with characters in \[\rep_{{\mathcal E}, \mathrm{u}}(\pi_1(\mathcal{O}), \PGL(n+1,\bR)).\] 
Also, we denote by $\Def_{{\mathcal E}}^s(\mathcal{O}) \subset \Def_{{\mathcal E}}(\mathcal{O})$
and $\Def_{{\mathcal E}, \mathrm{u}}^s(\mathcal{O}) \subset \Def_{{\mathcal E}, \mathrm{u}}(\mathcal{O})$ the subspaces of equivalence classes of real projective 
structures with stable and irreducible characters. 

\subsection{Oriented real projective structures} \label{subsec-orp} 

Recall that $\SL_\pm(n+1, \bR)$ is isomorphic to $\GL(n+1, \bR)/\bR^+$. 
Then this group acts on $\SI^n = \SI(\bR^{n+1})$.
We let $[v]$ denote the equivalence class of $v \in \bR^{n+1} -\{O\}$. 
There is a double covering map $\SI^n \ra \rpnn^n$ with the deck transformation group generated by
$\mathcal{A}$. 
This gives a projective structure on $\SI^n$. The group of projective automorphisms is identified with $\SL_\pm(n+1, \bR)$. \index{$\SL_\pm(n+1, \bR)$}


An $(\SI^n, \SL_\pm(n+1, \bR))$-structure on $\mathcal{O}$ is said to be an {\em oriented real projective structure} on $\mathcal{O}$. \index{real projective structure!oriented} 
We define $\Def_{\SI^n}(\mathcal{O})$ as the deformation space of $(\SI^n, \SL_\pm(n+1, \bR))$-structures on $\mathcal{O}$.

Again, we can define the {\em radial end structures} and {\em totally geodesic ideal boundary} for \index{end!radial end structure} 
oriented real projective structures and also horospherical end neighborhoods in obvious ways. \index{end!totally geodesic!ideal boundary} 
They correspond in the direct way in the following theorem also. 

\begin{theorem}\label{thm-doubledef} 
There is a one-to-one correspondence between the space of real projective structures on an orbifold $\orb$ 
with the space of oriented real projective structures on $\orb$. 
Moreover, a real projective diffeomorphism
of real projective orbifolds is an $(\SI^n, \SL_\pm(n+1, \bR))$-diffeomorphism of oriented real projective orbifolds
and vice versa. 
\end{theorem} 
\begin{proof} 
Straightforward. See p. 143 of Thurston \cite{Thbook}.
\end{proof}


\subsection{The  local homeomorphism theorems} 
For technical reasons, we will be assuming $\partial \orb = \emp$ in most cases. 
Here, we are not yet concerned with convexity of orbifolds.
The following map \hypertarget{term-hol}{$\hol$}, the so-called {\em Ehresmann-Thurston map}, is induced by sending $(\dev, h)$ to the conjugacy class of $h$ as 
isotopies preserve $h$:

\begin{theorem}[\cite{convMa}] \label{thm-A} 
Let $\mathcal{O}$ be a noncompact strongly tame real projective $n$-orbifold with radial ends or totally-geodesic ends of lens-type with 
markings and given types $\cR$ or $\cT$.  Assume $\partial \orb =\emp$. 
Then the following map is a local homeomorphism\,{\rm :}  
\[\hol:\Def_{{\mathcal E}, \mathrm{u}}^s(\mathcal{O}) \ra \rep_{{\mathcal E}, \mathrm{u}}^s(\pi_1(O), \PGL(n+1,\bR)).\]
\end{theorem}
Also, we define
 \[\rep_{\mathcal E}^s(\pi_1(\mathcal{O}), \SLpm), \rep^s_{\mathcal E, \mathrm{u}}(\pi_1(\mathcal{O}), \SLpm) \]  
 similarly to Section \ref{subsub-defspace}.
 
 By lifting $(\dev, h)$ by the method of Section \ref{subsec-orp}, 
we obtain that
 \[\hol:\Def_{{\mathcal E}, \mathrm{u}}^s(\mathcal{O}) \ra \rep_{{\mathcal E}, \mathrm{u}}^s(\pi_1(O), \SLpm)\]
is a local homeomorphism. 

\begin{remark}
The restrictions of end types are necessary for this theorem to hold. 
(See Goldman \cite{Goldman3}, Canary-Epstein-Green \cite{Canary}, Choi \cite{dgorb}, and Bergeron-Gelander \cite{BG04} for many versions of similar results.)
\end{remark}




\section{Convex real projective structures} \label{sec-convr}

\subsection{Metrics} \label{subsec-metrics}

Let $\Omega$ be a properly convex open domain. A line or a subspace of dimension-one 
in $\rpnn^n$ has a two-dimensional homogenous coordinate system. 
Define a metric by defining the distance for $p, q \in \Omega$, 
\[d_\Omega(p, q)= \log|[o,s,q,p]|\]
where $o$ and $s$ are 
endpoints of the maximal segment in $\Omega$ containing $p, q$
where $o, q$ separates $p, s$
and $[o,s,q,p]$ denotes the cross ratio. 

Given a properly convex real projective structure 
on ${\mathcal{O}}$, there is a \hypertarget{term-hilmetric}{Hilbert metric} 
which we denote by $d_{\torb}$ \index{$d_{\torb}$}
on $\tilde{\mathcal{O}}^{o}$.
Since the metric $d_{\torb}$ is invariant under the deck transformation group, 
we obtain a metric $d_{\orb}$ on $\orb$.
Assume that $K_i \ra K$ \hyperlink{term-geoc}{geometrically} for a sequence of properly convex domains $K_i$ and
a properly convex domain $K$.  
Suppose that two sequences of points $\{x_i| x_i \in K_i^o\}$ and $\{y_i| y_i \in K_i^o\}$ 
converge to $x, y \in K^o$ respectively. Since the end of a maximal segments always are in $\partial K_i$
and $\partial K_i \ra \partial K$ by Lemma \ref{lem:bdconv}, we obtain
\begin{equation} \label{eqn:HiHa}
d_{K_i^o}(x_i, y_i) \ra d_{K^o}(x, y)
\end{equation} 
holds.  

\subsection{Convexity and convex domains}\label{subsec-conv}

\begin{proposition}[Kuiper \cite{Kuiper}, Koszul \cite{Kos}, Vey \cite{Vey68}]\label{prop-projconv} $ $
\begin{itemize}
\item A strongly tame real projective orbifold is properly convex if and only if each developing map sends 
the universal cover to a properly convex open domain bounded in an affine subspace of $\rpnn^n$. 
\item If a strongly tame convex real projective orbifold 
is not properly convex, then
its holonomy homomorphism is virtually reducible.
\end{itemize}
\end{proposition}



\begin{proposition}[Corollary 2.13 of Benoist \cite{Ben3}]\label{prop-Benoist}  
Suppose that a discrete subgroup $\Gamma$ of $\PGL(n, \bR)$ \rlp resp. $\SL_{\pm}(n, \bR)$\rrp \,
 acts on a properly convex $(n-1)$-dimensional open domain $\Omega$ in $\rpnn^{n-1}$ \rlp resp. $\SI^{n-1}$\rrp \, so 
that $\Omega/\Gamma$ is compact. Then the following statements are equivalent. 
\begin{itemize} 
\item Every subgroup of finite index of $\Gamma$ has a finite center. 
 \item Every subgroup of finite index of $\Gamma$ has a trivial center. 
\item Every subgroup of finite index of $\Gamma$ is irreducible in $\PGL(n, \bR)$ \rlp resp. $\SL_{\pm}(n, \bR)$\rrp. 
That is, $\Gamma$ is strongly irreducible. 
\item The Zariski closure of $\Gamma$ is semisimple. 
\item $\Gamma$ does not contain a normal infinite nilpotent subgroup. 
\item $\Gamma$ does not contain a normal infinite abelian subgroup.
\end{itemize}
\end{proposition}

\begin{theorem}[Theorem 1.1 of Benoist \cite{Ben3}] \label{thm-Benoist} 
Let $\Gamma$ be a discrete subgroup of $\PGL(n, \bR)$ {\rm (}resp. $\SL_{\pm}(n, \bR)${\rm )} with a trivial virtual center. 
Suppose that a discrete subgroup $\Gamma$ of $\PGL(n, \bR)$ {\rm (}resp. $\SL_{\pm}(n, \bR)${\rm )} acts on 
a properly convex $(n-1)$-dimensional open domain $\Omega$ so 
that $\Omega/\Gamma$ is a compact orbifold.
Then every representation of a component of $\Hom(\Gamma, \PGL(n, \bR))$ {\rm (}resp. $\Hom(\Gamma, \SL_{\pm}(n, \bR)) ${\rm )} containing the inclusion 
representation also acts on a properly convex $(n-1)$-dimensional open domain cocompactly
and properly discontinuously. 
\end{theorem}

In general, a {\em join} of two convex sets $C_1$
and $C_2$ in an affine subspace $A$ of $\rpn$ is defined as 
\[ \{[t v_1 + (1-t) v_2]| v_i \in C_{C_i}, i= 1, 2, 0 \leq t \leq 1 \} \] 
where $C_{C_i}$ is a cone in $\bR^{n+1}$ corresponding to $C_i$, $i=1, 2$. 
The join is denoted by $C_1 * C_2$ in this paper. 

Given subspaces $V_1, \dots, V_m \subset \rpnn^n$ (resp.  $\subset \SI^n$)
that are from linear independent subspaces in $\bR^{n+1}$
and a subset $C_i \subset V_i$ for each $i$, 
we define a {\em strict join} of $m$ sets $C_1, \dots, C_m$
\[ C_1 * \cdots * C_m := \left\{ \left[ \sum_{i=1}^m t_i v_i \right] \, \bigg| \sum_{i=1}^m t_i = 1, 0 \leq  t_i \leq 1, v_i \in C_{C_i} \right\},\]
where $C_{C_i}$ is a cone in $\bR^{n+1}$ corresponding to $C_i$.  (Of course, this depends on the choices of $C_{C_i}$
up to $\mathcal{A}$.) 


A {\em cone-over} a strictly joined domain is one containing a strictly joined domain $A$ and 
is a union of segments from a cone-point $\not\in A$  to points of $A$.


\begin{proposition}[Theorem 1.1 of Benoist \cite{Ben3}] \label{prop-Ben2} Assume $n \geq 2$. 
Let $\Sigma$ be a closed $(n-1)$-dimensional properly convex projective orbifold 
and let $\Omega$ denote its universal cover in $\rpnn^{n-1}$ \rlp resp. in $\SI^{n-1}$\rrp.
Then 
\begin{itemize}
\item[(i)] $\Omega$ is projectively diffeomorphic to the interior of 
a strict join $K_1 * \cdots * K_{l_0}$ where $K_i$ is a properly convex open domain of dimension $n_i \geq 0$
corresponding to a convex open cone $C_i \subset \bR^{n_i+1}$. 
\item[(ii)] $\Omega$ is the image of the interior of $C_1 \oplus \cdots \oplus C_{l_{0}}$. 
\item[(iii)] The fundamental group 
$\pi_1(\Sigma)$ is virtually isomorphic to a cocompact subgroup of 
 $\bZ^{l_0-1} \times \Gamma_1 \times \cdots \times \Gamma_{l_0}$ for 
$l_0 -1 + \sum_{i=1}^{l_{0}} n_i = n$ with following properties\,{\rm :}
\begin{itemize}
\item Each $\Gamma_i$ has the property that each finite index subgroup has a trivial center. 
\item Each $\Gamma_j$ acts on $K_j^{o}$ cocompactly and the Zariski closure is 
a semisimple Lie group in $\PGL(n_j+1, \bR)$\, {\rm (}resp. in $\SL_\pm(n_j+1, \bR)${\rm )},
and acts trivially on $K_m$ for $m \ne j$. 
\item The subgroup corresponding to $\bZ^{l_0-1}$ acts trivially on each $K_j$.
\end{itemize} 
\end{itemize} 
\end{proposition} 

\subsection{The duality} \label{sec-dualty} 
We starts from linear duality. Let us choose the origin $O$ in $\bR^{n+1}$. 
Let $\Gamma$ be a group of linear transformations $\GL(n+1, \bR)$. 
Let $\Gamma^*$ be the {\em affine dual group} defined by $\{g^{\ast -1}| g \in \Gamma \}$ acting on 
the dual space $\bR^{n+1\ast}$. 
Suppose that $\Gamma$ acts on a properly convex cone $C$ in $\bR^{n+1}$ with the vertex $O$.

An open convex cone  $C^*$ in $\bR^{n+1, *}$  is {\em dual} to an open convex cone $C $ in $\bR^{n+1}$  if 
$C^* \subset \bR^{n+1 \ast}$ equals 
\[ \big\{\phi \in \bR^{n+1 \ast}\big|\, \phi| \clo(C) -\{O\} > 0 \big\}. \]
$C^*$ is a cone with vertex as the origin again. Note $(C^*)^* = C$. 

Let $\bR_{+}$ denote the set of positive real numbers. 
Now $\Gamma^*$ acts on $C^*$. Also, if $\Gamma$ acts cocompactly on $C$ if and only if $\Gamma^*$ acts on $C^*$ cocompactly. 
A {\em central dilatation extension} $\Gamma'$ of $\Gamma$ is the subgroup of $\GL(n+1, \bR)$ generated by $\Gamma$ and a dilatation
$s\Idd$  by a scalar $s \in \bR_+ -\{1\}$ \index{central dilatation extension} 
with the fixed $O$.
The dual of $\Gamma'$ is a central dilatation extension of $\Gamma^*$. 

 Given a subgroup $\Gamma$ in $\Pgl$, an {\em affine lift} in $\GL(n+1, \bR)$ is any subgroup that maps to $\Gamma$ isomorphically
 under the projection. 
 Given a subgroup $\Gamma$ in $\Pgl$, the dual group $\Gamma^*$ is the image in $\Pgl$ of the dual of 
 any affine lift of $\Gamma$. 

A properly convex open domain $\Omega$ in $\bP(\bR^{n+1})$ (resp. in $\SI(\bR^{n+1})$) is {\em dual} to a properly convex open domain
$\Omega^*$ in $\bP(\bR^{n+1 \ast})$ (resp. in $\SI(\bR^{n+1 \ast})$)  if $\Omega$ corresponds to an open convex cone $C$ 
and $\Omega^*$ to its dual $C^*$. We say that $\Omega^*$ is dual to $\Omega$. 
We also have $(\Omega^*)^* = \Omega$ and $\Omega$ is properly convex if and only if so is $\Omega^*$.  \index{dual!domain}

We call $\Gamma$ a {\em dividing group} if a central dilatational extension acts cocompactly on $C$. \index{divisible group}
$\Gamma$ is dividing if and only if so is $\Gamma^*$. (See \cite{Vin} and \cite{Ben1}). 

\begin{theorem}[Vinberg \cite{Vin3}] \label{thm-Vinberg} 
Let $\orb$ be a properly convex real projective orbifold of form $\Omega/\Gamma$. 
Let $\orb^{\ast}=\Omega^{\ast}/\Gamma^{\ast}$ be a properly convex real projective orbifold. 
Then $\orb^{\ast}$ is diffeomorphic to $\orb$. 
\end{theorem} 


Given a convex domain $\Omega$ in an affine subspace $A \subset \bR^{n}$, 
a {\em supporting hyperplane} $h$ at a point $x \in \Bd \Omega$ is 
a hyperplane so that a component of $A - h$ contains the interior of $\Omega$. 

A hyperspace is an element of $\rpnn^{n \ast}$ since it is represented as a $1$-form, 
and an element of $\rpnn^n$ can be considered as a hyperspace in $\rpnn^{n \ast}$. 
The following definition applies to $\Omega \subset \rpn$ (resp. $\SI(\bR^{n+1 \ast})$)  
and $\Omega^* \subset \rpnn^{n \ast}$ (resp. $\SI(\bR^{n+1\ast})$).
Given a properly convex domain $\Omega$, we define the {\em augmented boundary} of $\Omega$ \index{augmented boundary} 
\[\Bd^{\Ag} \Omega  := \big\{ (x, h) \big| x \in \Bd \Omega, h \hbox{ is a supporting hyperplane of } \Omega,  h \ni x \big\} .\] 

\begin{remark} 
For open properly convex domains $\Omega_1$ and $\Omega_2$, 
we have 
\begin{equation}\label{eqn:dualinc}
\Omega_1 \subset \Omega_2 \hbox{ if and only if } \Omega_2^* \subset \Omega_1^*. 
\end{equation}
\end{remark}

The following standard results are proved in Section 3 of \cite{End1}. 
We will call the homeomorphism below as the {\em duality map}. \index{duality map}
\begin{proposition} \label{prop-duality}
Suppose that $\Omega \subset \rpn$\, \rlp resp. $\SI(\bR^{n+1\ast})$\rrp \, and 
its dual $\Omega^* \subset \rpnn^{n \ast}$\, \rlp resp. $\SI(\bR^{n+1\ast})$\rrp \,
are properly convex domains. 
\begin{itemize}  
\item[(i)] There is a proper quotient map $\Pi_{\Ag}: \Bd^{\Ag} \Omega \ra \Bd \Omega$
given by sending $(x, h)$ to $x$. 
\item[(ii)] Let $\Gamma$ act on properly discontinuously $\Omega$ if and only if so acts
$\Gamma^*$ on $\Omega^*$.
\item[(iii)] There exists a homeomorphism 
\[ {\mathcal{D}}: \Bd^{\Ag} \Omega \leftrightarrow \Bd^{\Ag} \Omega^* \] 
given by sending $(x, h)$ to $(h, x)$. 
\item[(iv)] Let $A \subset \Bd^{\Ag} \Omega$ be a subspace and $A^*\subset \Bd^{\Ag} \Omega^*$
be the corresponding dual subspace $\mathcal{D}(A)$. If a group $\Gamma$ acts properly discontinuously on $A$ 
if and only if $\Gamma^*$ so acts on $A^*$. 
\end{itemize} 
\end{proposition} 

Given a convex domain $\Omega$, we denote by $R_p(\Omega)$ the space of directions of 
open rays in $\Omega$ from $p$ in $\SI^{n-1}_p$.  

\begin{proposition}\label{prop-example}
Let $\Omega^*$ be the dual of a properly convex domain $\Omega$.
Then 
\begin{itemize}
\item[(i)] $\Bd \Omega$ is $C^1$ and strictly convex if and only if $\Bd \Omega^*$  is $C^1$ and strictly convex.
\item[(ii)] $\Omega$ is a horospherical orbifold if and only if so is $\Omega^*$. 
\item[(iii)] Let $p \in \Bd \Omega$. Then
$\mathcal{D}$ sends  in a one-to-one and onto manner
\[\{(p, h)| h \hbox{ is a supporting hyperplane of $\Omega$ at $p$}\}\]
to $\{(h^*, p^*)| h^* \in D = p^* \cap \Bd \Omega^* \}$ where $D$ is a properly convex set in $\Bd \Omega$.   
\item[(iv)] $\Bd \Omega^*$ contains a properly convex domain $D = P \cap \Bd \Omega^*$ open in a totally geodesic hyperplane $P$ 
if and only if $\Bd \Omega$ contains 
a vertex $p$ with $R_p(\Omega)$ a properly convex domain. Moreover, $D^o$ and $R_p(\Omega)$ are properly
convex and are projectively diffeomorphic to dual domains in $\rpnn^{n-1}$. 
\end{itemize}
\end{proposition}



\section{The end theory} \label{sec-endt}

We will now discuss in detail the end theory.  The following is simply the notions useful in 
relative hyperbolic group theory as can be found in Bowditch \cite{Bowditch}.

\subsection{ p-ends, p-end neighborhoods, and p-end fundamental groups } \label{sub-pend}

Let $\orb$ be a real projective orbifold with the universal cover $\torb$ and the covering map $p_{\torb}$. 
Each end neighborhood $U$, diffeomorphic to $S_{E} \times (0, 1)$, 
of an end $E$ lifts to a connected open set 
$\tilde U$ in $\torb$. 
A subgroup $\Gamma_{\tilde U}$ of $\Gamma$ acts on $\tilde U$ where 
\[p_{\torb}^{-1}(U) = \bigcup_{g\in \pi_1(\orb)} g(\tilde U).\] 
Each component
$\tilde U$ is said to be a {\em proper pseudo-end neighborhood}. 
\begin{itemize} 
\item An {\em exiting sequence} of sets $U_{1}, U_{2}, \cdots $ in $\torb$ is 
a sequence so that for each compact subset $K$ of $\orb$
there exists an integer $N$ satisfying $p_{\torb}^{-1}(K) \cap U_i = \emp$ for $i > N$.  
\item A {\em pseudo-end sequence} is an exiting sequence of proper pseudo-end neighborhoods 
\[\{U_{i}|i=1, 2, 3, \dots\}, \hbox{ where } U_{i+1} \subset U_{i} \hbox{ for every } i.\]
\item Two pseudo-end sequences $\{ U_{i}\}$ and $\{V_{j}\}$ are {\em compatible} if 
for each $i$, there exists $J$ such that $ V_{j} \subset U_{i}$ for every $j$, $j > J$
and conversely for each $j$, there exists $I$ such that $U_{i} \subset V_{j}$ for every $i$,  $i > I$. 
\item A compatibility class of a proper pseudo-end sequence is called a \hypertarget{term-pend}{{\em pseudo-end}} of $\torb$.
Each of these corresponds to an end of $\orb$ under the universal covering map $p_{\orb}$.
\item For a pseudo-end $\tilde E$ of $\torb$, we denote by $\Gamma_{\tilde E}$ the subgroup $\Gamma_{\tilde U}$ where 
$U$ and $\tilde U$ is as above. We call $\Gamma_{\tilde E}$ a \hypertarget{term-pendfund}{{\em pseudo-end fundamental group}}.
We will also denote it by $\pi_{1}(\tilde E)$. 
\item \hypertarget{term-pendnhbd}{A {\em pseudo-end neighborhood }} $U$ of a pseudo-end $\tilde E$ is a $\Gamma_{\tilde E}$-invariant open set containing 
a proper pseudo-end neighborhood of $\tilde E$. 
A proper pseudo-end neighborhood is an example. 
\end{itemize}
\hypertarget{term-p}{(From now on, we will replace ``pseudo-end'' with the abbreviation ``p-end''.) }

\begin{proposition} \label{prop:endf} 
The p-end fundamental group $\Gamma_{\tilde E}$ is independent of the choice of $U$. 
\end{proposition}
\begin{proof}
Given $U$ and $U'$ that are end-neighborhoods for an end $E$, 
let $\tilde U$ and $\tilde U'$ be p-end neighborhoods for a p-end $\tilde E$
that are components of $p^{-1}(U)$ and $p^{-1}(U')$ respectively. 
Let $\tilde U''$ be the component of $p^{-1}(U'')$ that is a p-end neighborhood of $\tilde E$. 
Then $\Gamma_{\tilde U''}$ injects into $\Gamma_{\tilde U}$
since both are subgroups of $\Gamma$. 
Any $\mathcal{G}$-path in $U$ in the sense of Bridson-Haefliger \cite{BH} is homotopic 
to a $\mathcal{G}$-path in $U''$ by a translation in the $I$-factor. 
Thus, $\pi_{1}(U') \ra \pi_{1}(U)$ is surjective.  
Since $\tilde U$ is connected, any element $\gamma$ of 
$\Gamma_{\tilde U}$ is represented by a $\mathcal{G}$-path connecting 
$x_{0}$ to $\gamma(x_{0})$.  (See Example 3.7 in Chapter III.$\mathcal{G}$ of \cite{BH}.) 
Thus, $\Gamma_{\tilde U}$ 
is isomorphic to the image of $\pi_{1}(U) \ra \pi_{1}(\orb)$. 
Since $\Gamma_{\tilde U''}$ is surjective to the image of of $\pi_{1}(U'') \ra \pi_{1}(\orb)$, 
it follows that $\Gamma_{\tilde U''}$ is isomorphic to $\Gamma_{\tilde U}$. 
\end{proof}

Let $\orb$ be a strongly tame real projective orbifold.
We give each end of $\orb$ a marking. 
We fix a developing map $\dev:\torb \ra \rpn$ in this subsection.


A {\em ray} from a point $v$ of $\rpn$ is a segment with endpoint equal to $v$ 
oriented away from $v$. 

Let $E$ be an \hyperlink{term-Rend}{R-end} of $\orb$. 
\begin{itemize}
\item Let $\tilde E$ denote a p-R-end corresponding to $E$ and $U$ denote a p-R-end neighborhood of $\tilde E$
with a radial foliation $\mathcal{F}$ induced from the end marking on an end neighborhood of $E$.
\item Two radial leaves of equivalent radial foliations of proper p-end neighborhoods of $\tilde E$
are {\em equivalent} if they agree on a proper p-end neighborhood of $\tilde E$.
\item Let $x$ be the common end point of the images under the developing map of  leaves of $\mathcal{F}$. 
We call $x$ the {\hypertarget{term-pendvertex}{{\em p-end vertex}}} of $\torb$. 
$x$ will be denoted by $v_{\tilde E}$ if its associated p-end neighborhood corresponds to a p-end $\tilde E$. \index{end!p-end vertex} 
\item Let $\SI^{n-1}_{v_{\tilde E}}$ denote \hypertarget{detail-rays}{the space of equivalence classes of rays} 
from $v_{\tilde E}$ diffeomorphic to an $(n-1)$-sphere \index{$\SI^{n-1}_{v_{\tilde E}}$}
where $\pi_1({\tilde E})$ acts as a group of projective automorphisms.  
Here, $\pi_1({\tilde E})$ acts on $v_{\tilde E}$ and sends leaves to leaves in $U_1$. 
\item 
We denote by $R_{v_{\tilde E}}(\torb)= \tilde S_{\tilde E}$ as the following space
\[\big\{[l]\big|\, l \subset \torb, \dev(l) \hbox{ is a ray from } v_{\tilde E}\big\},\]
which is an $(n-1)$-dimensional open manifold. 
\item The map $\dev$ induces an immersion
\[ \tilde S_{\tilde E} \ra \SI^{n-1}_{v_{\tilde E}}.\]
Also, $\Gamma_{\tilde E}$ projectively acts on $\tilde S_{\tilde E}$ 
by $g([l]) = [g(l)]$ for each leaf $l$ and $g \in \Gamma_{\tilde E}$. 
\index{$R_{v_{\tilde E}}(\torb)$} \index{$\tilde S_{\tilde E}$}
\item Recall that $\tilde S_{\tilde E}/\Gamma_{\tilde E}$ is diffeomorphic to the end orbifold 
denoted by $S_E$.  Thus, $S_{E}$ has a convex real projective structure. 
(However, the projective structure and the differential topology on $S_{E}$ does depend on the end markings.)
\end{itemize}

Given a \hyperlink{term-Tend}{T-end} of $\orb$ and an end neighborhood $U$ of the product form $S_{E} \times [0, 1)$
with a compactification by a totally geodesic orbifold $S_{E}$, 
we take a component $U_1$ of $p^{-1}(U)$ and a convex domain $\tilde S_{\tilde E}$  
developing into a totally geodesic hypersurface $P$  under $\dev$. 
Here $\tilde E$ is the p-end corresponding to $E$ and $U_1$.
$\Gamma_{\tilde E}$ acts on $U_{1}$ and hence on $\tilde S_{\tilde E}$. 
Again ${\tilde S_{\tilde E}}/\Gamma_{\tilde E}$ is projectively diffeomorphic to the {\em end orbifold} to be denote by $S_E$ again. 
We call $\tilde S_{\tilde E}$ the {\em p-ideal boundary component} of $\torb$. \index{end!p-ideal boundary component} 
Generalizing further an open subset $U$ of $\torb$ containing a proper p-end-neighborhood of $\tilde E$, 
where $\pi_1(\tilde E)$ acts on, is said to be a {\em p-end neighborhood}. \index{end!p-end neighborhood}

\subsection{The admissible groups} \label{sub-gpadmissible}

If every subgroup of finite index of a group $\Gamma$ has a finite center, $\Gamma$ is said to be a {\em virtual center-free group}.
An {\em admissible group} $G$   acting on projective $\SI^{n-1}$ is a finite extension of the finite product  \index{admissible group}
$\bZ^{l-1} \times \Gamma_1 \times \cdots \times \Gamma_l$ 
for infinite hyperbolic or trivial groups $\Gamma_i$ 
with following properties:
\begin{itemize} 
\item $G$ acts on a properly convex domain of form $K_{1}\ast \cdots \ast K_{l}$ for a strictly convex domain $K_{j} \subset \SI^{n-1}$ for each $j$, $j = 1, \dots, l$,
in an $n_{j}$-dimensional subspace of $\SI^{n-1}$, $0 \leq n_{j} \leq n-1,$
\item $\Gamma_{j}$ is the restriction of $G$ to each $K_{j}$ 
and extended on $\SI^{n-1}$ to act trivially on $K_{m}$ for $m \ne j$
where $K_{j}^{o}/\Gamma_{j}$ is an orbifold of dimension $n_{j}$. 
\item $\bZ^{l-1}$ acts trivially on each $K_{j}$ and is a virtual center of $G$. 
\end{itemize}
This is strictly stronger than the conclusion of Proposition \ref{prop-Ben2} of Benoist and is needed for now.
Here, we conjecture that we do not need the stronger condition but the conclusion of Proposition \ref{prop-Ben2} 
is enough assumption for everything in this paper. 
(See also Example 5.5.3 of \cite{DW2} as pointed out by M. Kapovich. )

We have $l=1$ if and only if the end fundamental group is hyperbolic or is trivial. 
If our orbifold has a complete hyperbolic structure, then end fundamental groups are virtually free abelian. 
%


\subsection{The admissible ends} \label{sub-admissible} 

\begin{figure}[h]
\centerline{\includegraphics[height=7cm]{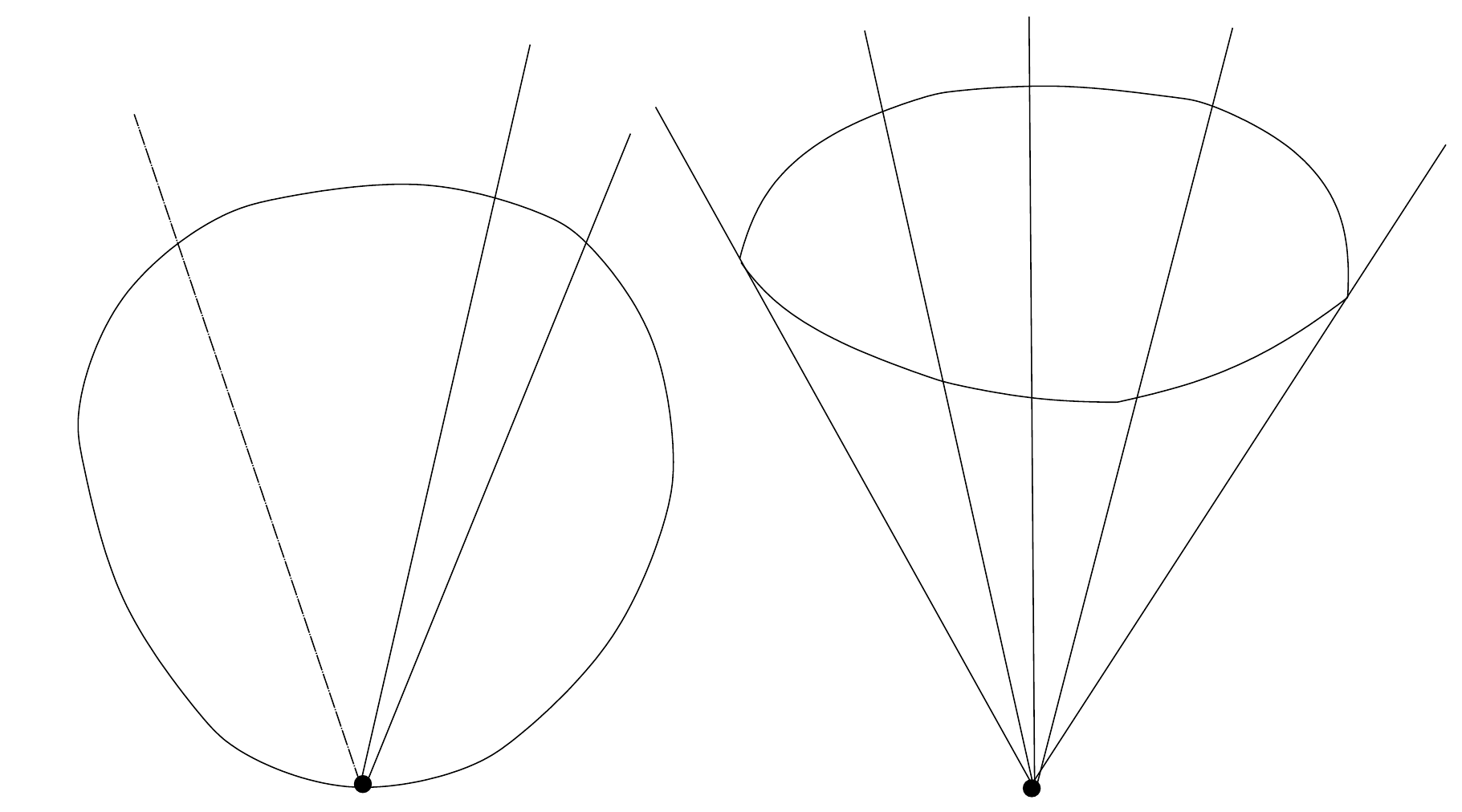}}
\caption{The universal covers of horospherical and lens shaped ends. The radial lines form cone-structures.}
\label{fig:lens}
\end{figure}

Let $\orb$ be a convex real projective orbifold with the universal cover 
$\torb$. 
\begin{itemize} 
\item A \hypertarget{term-cone}{{\em cone}} 
over a point $x$ and a set $A$ in an affine subspace of $\rpnn^n$ (resp. in $\SI^n$), $x \not\in \clo(A)$ is the set \index{cone} 
given by $x\ast A$ in $\rpnn^n$ (resp. in $\SI^n$).
\item Take a cone $C:= \{x\} \ast L $ over a \hyperlink{term-lens}{lens} $L$ 
and a point $x$, $x\not\in \clo(L)$, 
so that every maximal segment $l$ from $x$ in $C$ ends in one component $\partial_1 L$ of $\partial L$
and meets $\partial_{1} L$ and $\partial_{2} L$ exactly once. 
A \hypertarget{term-lenscone}{{\em lens-cone}} is $C - \{x\}$. 
\item Take a cone $C:= \{x\} \ast L $ over a \hyperlink{term-glens}{generalized lens} $L$
with the same properties as above where a nonsmooth component has to be in the boundary of the cone. 
A \hypertarget{term-glensc}{{\em generalized lens-cone}} is $C - \{x\}$. 
\item For two components $A_1$ and $A_2$ of $\partial L$ for $L$ as above in the lens-cone, 
$A_1$ is called a {\em top hypersurface} if it is in $\Bd (\{x\}\ast L)$ and $A_2$ is then called a {\em bottom hypersurface}.
\item A (generalized) \hypertarget{term-lensin}{{\em lens}} of 
a (generalized) lens-cone $C$ is the lens-shaped domain $A$ so that $C = \{ x \} \ast A -\{x\}$ for a point $x \not\in \clo(A)$ \index{lens} 
and with the properties in the first item. 
\item A \hypertarget{term-tgd}{{\em totally-geodesic subdomain}} is a convex domain in a hyperspace. 
\item A \hypertarget{term-cot}{{\em cone-over} a totally-geodesic open domain} $A$ is $\{x\} \ast A - \{x\}$ for
the cone $\{x\}\ast A$ over a point $x$ not in the hyperspace. \index{cone} 
\end{itemize} 
(See Figure \ref{fig:lens}.)
We will also call a real projective orbifold with boundary to be 
\begin{itemize}
\item a {\em lens-cone} or \index{lens-cone} 
\item a {\em lens}, provided it is compact, \index{lens} 
\end{itemize} if it is covered by such domains and is diffeomorphic to a closed $(n-1)$-orbifold times an interval.



We introduce some relevant adjectives: 
Let $S_E$ be an $(n-1)$-dimensional end orbifold corresponding to a p-end $\tilde E$, and
let $\mu$ be a holonomy homomorphism 
\[\pi_1(\tilde E) \ra \PGL(n+1, \bR) \hbox{ (resp. } \SL_{\pm}(n+1, \bR) \hbox{)}  \] 
restricted from that of $\orb$. 
\begin{itemize} 
\item Suppose that $\mu(\pi_1(\tilde E))$
acts on a (generalized) lens-shaped domain $K$ in $\rpnn^n$ (resp. in $\SI^n$) with boundary
a union of two open $(n-1)$-cells $A_1$ and $A_2$ and $\pi_1(\tilde E)$ acts properly on $A_1$ and $A_2$. 
Then $\mu$ is said to be a \hypertarget{term-lensrep}{({\em generalized}\,) {\em lens-shaped} representation} for $\tilde E$. 
\item $\mu$ is a \hypertarget{term-totgrep}{{\em totally-geodesic} representation} if 
$\mu(\pi_1(\tilde E))$ acts on a totally-geodesic subdomain. 
\item If $\mu(\pi_1(\tilde E))$ acts on a horoball $K$, then 
$\mu$ is said to be a  \hypertarget{term-hororep}{{\em horospherical} representation}. 
In this case, it follows $\Bd K - \partial K=\{v_{\tilde E} \}$ for the p-end vertex $v_{\tilde E}$ of $\tilde E$. 
\item If $\mu(\pi_1(\tilde E))$ acts on a strictly joined domain,
then $\mu$ is said to be a \hypertarget{term-strjrep}{{\em  strictly joined} representation}. 
\end{itemize} 



Let $C'$ be a generalized lens and $L:= \{v_{\tilde E}\} \ast C' - \{v_{\tilde E}\}$ be a \hyperlink{term-glensc}{generalized lens-cone} over $C'$.
A \hypertarget{term-cpenbd}{{\em concave p-end-neighborhood}} is an imbedded p-end neighborhood of form
 $L - C'$. 
 \index{end!p-end neighborhood!concave}  

\begin{definition}{(Admissible ends)} \label{defn-admissible} 
Let $\orb$ be a real projective orbifold with the universal cover $\torb$. 
Let $E$ be an R-end of $\orb$ and $\tilde E$ be the corresponding p-end with the p-end fundamental group 
$\pi_1(\tilde E)$. 
\begin{itemize} 
\item We say that the radial end $E$ of $\orb$ is of \hypertarget{term-lensend}{{\em lens-type}} if $\tilde E$ has a p-end neighborhood that 
is a lens-cone of form $L\ast \{v_{\tilde E}\} -  \{v_{\tilde E}\}$ and $\pi_{1}(\tilde E)$ acts on for its lens $L$. 
\item $E$ is of \hypertarget{term-genlensend}{{\em generalized lens-type}} if $\tilde E$ has a concave p-end neighborhood. 
Equivalently, a p-end neighborhood of $\tilde E$ is the interior of a generalized lens-cone of form $L\ast \{v_{\tilde E}\} -  \{v_{\tilde E}\}$ 
and $\pi_1(\tilde E)$ acts on the generalized lens $L$. 
\end{itemize} 

A p-R-end $\tilde E$ is \hypertarget{term-add}{{\em admissible}} if the p-end fundamental group acts on $\tilde S_{\tilde E}$
as an admissible group and 
if $\tilde E$ is a \hyperlink{term-horo}{horospherical} or \hyperlink{term-lensend}{lens-type p-R-end}. \index{end!admissible} \index{end!generalized admissible} 
\end{definition} \index{end!radial lens-type} \index{end!horospherical} \index{end!totally geodesic!lens} 


A T-end $E$ is of {\em lens-type} if $E$ satisfies the \hyperlink{term-lensT}{lens-condition} that the ideal boundary 
end orbifold $S_{E}$ has a lens-neighborhood $L$  in an ambient real projective orbifold containing $\orb$. 
For a component $C_1$ of $L - S_{E}$ inside $\orb$, 
$C_1 \cup S_{E}$ is said to be the \hypertarget{term-osed}{{\em one-sided end neighborhood}} 
of $S_{E}$. \index{one-sided end neighborhood} 
Given a p-end $\tilde E$, the orbifold 
$S_{E}$ is covered by a domain $\tilde S_{\tilde E}$ in the boundary of 
p-end neighborhood corresponding $\tilde E$ in a hyperspace.

A T-end is {\em admissible } if it is of lens-type and the p-end fundamental group for a p-end $\tilde E$ 
acts on the ideal boundary $\tilde S_{\tilde E}$ as an admissible group. 

A p-end is  \hypertarget{term-addg}{{\em admissible in a generalized sense}} if it is admissible or is 
a \hyperlink{term-genlensend}{generalized lens-type p-R-end}.

\begin{example}
A model of a lens-type R-end can be made by a positive diagonal group acting on 
the standard simplex $T$ in $\mathbbm{RP}^{3}$.
We take a vertex $v$ to be $[1,0,0,0]$ 
and we choose an abelian group $G$ of rank $3$ acting on $T$
and properly and freely on the interior $F$ of the side of $T$ opposite $v$.
We choose $G$ so that the eigenvalue at $v$ is not the largest or the smallest one for $g \in \Delta$
for the Zariski closure $\Delta$ of $G$. 
$(\{v\} \ast F -\{v\})/G$ is an end neighborhood of an ambient orbifold.  
The existence of lens follows by considering orbits of points under $\Delta$. 
(This follows by Theorem 5.1 of \cite{End1} since we can show that the uniform middle eigenvalue condition holds.
See also  Ballas \cite{Ballas}, \cite{Ballas2}, and Ballas-Danciger-Lee \cite{BDL} which include many graphics for ends.) 

Example \ref{exmp-endv} give these examples by Proposition  4.6 of \cite{End1} or more generally by Theorem \ref{thm:niceend}; 
that is,  we show that these have to be admissible lens-type R-ends or horospherical R-ends.
(Note also that these properties of the examples will hold during deformations as we will show later in Section \ref{sub-open}.) 
\end{example}



\section{The relative hyperbolicity of $\pi_{1}(\orb)$ and the strict convexity} \label{sec-relhyp}

\subsection{SPC-structures and its properties} \label{sub-spc}

\begin{definition} \label{defn:IE}
For a strongly tame orbifold $\mathcal{O}$, 
\begin{itemize}
\item[(IE)] $\orb$ or $\pi_1(\orb)$ satisfies the {\em infinite-index end fundamental group condition} (IE) \index{end!condition!IE}
if $[\pi_1(\mathcal O): \pi_1(E)] = \infty$ for the end fundamental group $\pi_1(E)$ of each end $E$. 
\item[(NA)] $\orb$ or $\pi_1(\orb)$ satisfies the {\em nonannular property} (NA) \index{end!condition!NA} 
if \[\pi_{1}(\tilde E_{1}) \cap \pi_{1}(\tilde E_{2}) \] is finite for two distinct p-ends $\tilde E_{1}, \tilde E_{2}$ of $\torb$, 
and a free abelian group of rank $2$ is conjugate to a subgroup of $\pi_{1}(E)$ for some end $E$. 
\end{itemize} 
\end{definition} 
(NA) implies that $\orb$ contains no essential torus and also that 
$\pi_1(E)$ contains every element $g \in \pi_1(\orb)$ normalizing $\langle h \rangle$ for 
an infinite order $h \in \pi_1(E)$ for an end fundamental group $\pi_1(E)$ of an end $E$. 
These conditions are satisfied by complete hyperbolic manifolds with cusps
and are group theoretical properties with respect to the end groups.

\begin{definition} 
An \hypertarget{term-spc}{{\em SPC-structure}} or {\em stable properly-convex real projective structure} on an $n$-orbifold 
is a real projective structure so that the orbifold is projectively diffeomorphic to a quotient orbifold of 
a properly convex domain in $\rpnn^n$ by a discrete group
of projective automorphisms that is \hyperlink{term-st}{stable} and  irreducible.
\end{definition}

\begin{definition}\label{defn:strict} 
Suppose that $\mathcal{O}$ has an SPC-structure. Let $\tilde U$ be 
the inverse image in $\tilde{\mathcal{O}}$ of the union $U$ of some choice of a collection of disjoint end neighborhoods of $\orb$
with compact $\clo(U)$. \index{SPC-structure}
If every straight arc and every non-$C^1$-point   in $\Bd \tilde{\mathcal{O}}$ 
are contained in the closure of a component of $\tilde U$, 
then $\mathcal{O}$ is said to be {\em strictly convex} with respect to the collection of the ends.  \index{convex!strictly}
And $\mathcal{O}$ is also said to have a \hypertarget{term-sspc}{{\em strict SPC-structure}} with respect to the collection of ends. \index{SPC-structure!strict}
\end{definition}
Notice that the definition depends on the choice of $U$. However, we will show that 
if each component $U$ is required to be of lens-type or \hyperlink{term-horo}{horospherical}, then
we show that the definition is independent of $U$ in \cite{convMa}. 

\begin{theorem}\label{thm-sSPC} 
Let $\orb$ be a noncompact strongly tame properly convex real projective manifold with 
\hyperlink{term-addg}{generalized admissible ends} 
and satisfies {\rm (IE)} and {\rm (NA)}.  
Then any finite-index subgroup of the holonomy group is strongly  irreducible and is not 
contained in a proper \hyperlink{term-pgroup}{parabolic subgroup} of $\PGL(n+1, \bR)$  {\rm (}resp.  $\SLpm${\rm ).}
\end{theorem} 

For proof, see Section 4 of \cite{convMa}.

\subsection{Bowditch's method}\label{sub-Bowditch} 

There are results proved by Cooper, Long, and Tillman \cite{CLT3} and Crampon and Marquis \cite{CM} 
similar to below. However, the ends have to be horospherical in their work. 
We will use Bowditch's result \cite{Bowditch} to show 

\begin{theorem}\label{thm-relhyp}
Let $\mathcal{O}$ be a noncompact strongly tame 
\hyperlink{term-sspc}{strict SPC-orbifold} with \hyperlink{term-addg}{generalized admissible ends} 
$E_1, \dots, E_k$ and satisfies {\rm (IE)} and {\rm (NA)}. 
Assume $\partial \orb =\emp$. 
Let $\tilde U_i$ be the inverse image $U_i$ in $\torb$ for a mutually disjoint collection of
 neighborhoods $U_i$ of the ends $E_i$ for $i=1, \dots, k$. 
Then 
\begin{itemize} 
\item $\pi_1(\mathcal{O})$ is relatively hyperbolic with respect to 
the end fundamental groups \[\pi_1(E_1), \ldots , \pi_1(E_k).\] 
Hence $\orb$ is relatively hyperbolic with 
respect to $U_1 \cup \cdots \cup  U_k$. 
\item If $\pi_1(E_{l+1}), \ldots , \pi_1(E_k)$ are hyperbolic for some $1 \leq  l \leq k$
{\rm (}possibly some of the hyperbolic ones{\rm )},
then $\pi_1(\mathcal{O})$ is relatively hyperbolic with respect to the end fundamental group 
$\pi_1(E_1), \dots, \pi_1(E_{l})$. 
\end{itemize}
\end{theorem}
For definitions and results on relative hyperbolicity of metric spaces, see Bowditch \cite{Bowditch} or Farb \cite{Farb}.

The idea for proof is as follows: 
Let $U$ be a union of end neighborhoods of $\orb$ diffeomorphic to 
an orbifold times an interval. $\orb - U$ is a compact orbifold with boundary.  
We contract $\clo(C) \cap \Bd \torb$ for each component $C$ of $p^{-1}(U)$ to 
a singleton to obtain a quotient space $X$. 
Then $X$ is homeomorphic to a compact metric space, i.e., a compactum.
We demonstrate that the axioms of Bowditch are satisfied 
by analyzing the triples of points in $X$ in \cite{convMa}. 

\subsection{Converse}

The converse to Theorem \ref{thm-relhyp} is as follows:

\begin{theorem} \label{thm-converse}
Let $\mathcal{O}$ be a noncompact strongly tame properly convex real projective orbifold with \hyperlink{term-addg}{generalized admissible ends} 
and satisfies {\rm (IE)} and {\rm (NA)}.  Assume $\partial \orb =\emp$. 
Suppose that $\pi_1(\mathcal{O})$ is a relatively hyperbolic group with respect to 
the admissible end groups $\pi_1(E_1), ..., \pi_1(E_k)$. 
Then $\mathcal{O}$ is \hyperlink{term-sspc}{strictly SPC} with respect to the admissible ends $E_1, \dots, E_k$. 
\end{theorem}

Let $U$ be as in the above section, and let $\tilde U = p^{-1}(U) \subset \torb$. 
To give some idea of the proof, we take any segment in $\Bd \torb$
not contained in any component of $\clo(\tilde U)\cap \Bd \torb$.
Then we find a triangle $T$ with $\partial T \subset \Bd \torb$ using a sequence of points
converging to an interior point of the segment.
Also, we construct so that $\partial T$ is not in the closure of any p-end neighborhood.

Let us recall standard definitions in Section 3.1 of Drutu-Sapir \cite{DuSa}. 
An {\em ultrafilter $\omega$} is a finite additive measure on $P(\bN)$ of $\bN$ so that
each subset has either measure $0$ or $1$ and all finite sets have measure $0$. 
If a property $P(n)$ holds for all $n$ from a set with measure $1$, we say that $P(n)$ holds {\em $\omega$-almost surely}. \index{ultrafilter}

Let $(X, d_X)$ be a metric space.
Let $\omega$ be a nonprincipal ultrafilter over the set $\bN$ of natural numbers. 
For a sequence $(x_i)_{i \in \bN}$ of points of $X$, its {\em $\omega$-limit} is $x \in X$ if for every neighborhood $U$ of $x$
the property that $x_i \in U$ holds $\omega$-almost surely. 

An {\em ultraproduct} $\prod X_n/\omega$ of a sequence of sets $(X_n)_{n \in \bN}$ is the set of 
the equivalence classes of sequences $(x_n)$ where $(x_n) \sim (y_n)$ if $x_n  = y_n$ holds for $\omega$-almost surely. 

Given a sequence of metric spaces $(X_n, d_n)$, consider the ultraproduct $\prod X_n$ and an observation point $e=(e_n)$. 
Let $D(x, y) = \lim_{\omega} d_n(x_n, y_n)$. Let $\prod_e X_n/\omega$ denote the set of equivalence classes of sequences of bounded distances from $e$.
The {\em $\omega$-limit $\lim^{\omega} (X_n)_e$} is the metric space obtained from $\prod_e X_n/\omega$ by identifying  \index{asymptotic cone} 
all pair of points $x, y$ with $D(x, y) =0$. 

Given an ultrafilter $\omega$ over the set $\bN$ of natural numbers, an observation point $e=(e_i)^\omega$, and 
sequence of numbers $\delta= (\delta_i)_{i\in \bN}$ satisfying $\lim_\omega \delta_i = \infty$, the $\omega$-limit 
$\lim^\omega (X, d_X/\delta_i)_e$ is called the {\em asymptotic cone} of $X$. (See \cite{Gr1}, \cite{Gr2} and Definitions 
3.3 to 3.8 in \cite{DuSa}.) We denote it by $Con^\omega(X, e, \delta)$. 

For a sequence $(A_n)$ of subsets $A_n$ of $X$, we denote by $\lim^\omega(A_n)$ the subset of $Con^\omega(X, e, \delta)$ 
that consists of all elements $(x_n)$ where $x_n \in A_n$ $\omega$-almost surely. 
The asymptotic cone is always complete and $\lim^\omega(A_n)$ is closed.


Next, we choose a nonprincipal ultrafilter $\omega$ and a sequence $l_{i} \ra \infty$.
We use the $\omega$-limit $\torb_{\infty}$ of $\frac{1}{l_{i}} d_{\torb}$ on $\torb$
with a constant base point $e_{i} = e \in \torb$. 
This turns out to be a tree-graded space in the sense of Drutu and Sapir \cite{DuSa}. 
Let $T_{i}^{o}$ be $T^{o}$ with the metric  $\frac{1}{l_{i}} d_{\torb}|T^{o}$, 
which is a hex metric of de la Harpe \cite{Harpe}. 

The inverse image in $U$ of $\torb$ of the union of disjoint end neighborhoods of $\orb$. 
A piece in the limit tree-graded space is a limit of a sequence of components of $U$
by Proposition 7.26 of \cite{DuSa}. 
The sequence $\{T^{o}_{i} \subset \torb\}$  converges to a triangle $T_{\infty}$
with the hex metric and we show that $T_{\infty}$ is not contained in a piece by a geometric argument. 
However, a triangle with a hex metric cannot be divided into more than one piece.  


\subsection{Strict SPC-structures deform to strict SPC-structures.}

By above Theorems \ref{thm-relhyp} and \ref{thm-converse},  the property of strictness of the SPC-structures is topological. 
Hence, the strictness is a stable property among the set of the SPC-structures. 

\begin{theorem}\label{thm-relhyp1}
Let ${\mathcal{O}}$ denote a noncompact strongly tame strict \hyperlink{term-spc}{SPC-orbifold} with admissible ends 
and satisfies {\rm (IE)} and {\rm (NA)}.  Assume $\partial \orb =\emp$. 
Let \[E_1, \dots, E_n, E_{n+1}, \dots, E_k\] be the ends of $\orb$
where $E_{n+1}, \dots, E_k$ are some or all of the hyperbolic ends. 
\begin{itemize}
\item Given a deformation through SPC-structures with \hyperlink{term-addg}{generalized admissible ends} of a strict SPC-orbifold with respect to admissible ends 
$E_1, \dots, E_k$ to an SPC-structure with \hyperlink{term-addg}{generalized admissible end}, the SPC-structures remain strictly SPC with respect to $E_1, \dots, E_k$. 
\item Given a deformation through SPC-structures with \hyperlink{term-addg}{generalized admissible ends} of a strict SPC-orbifold with respect to $E_1, \dots, E_n$ 
to an SPC-structure with \hyperlink{term-addg}{generalized admissible end},  the SPC-structures remain strictly SPC 
with respect to admissible ends $E_1, \dots, E_n$. 
\end{itemize}
\end{theorem}

\section{The openness and closedness in character varieties} \label{sec-clopen} 

We will now begin to discuss the main aim of this papers. This is to identify the deformation spaces of
convex real projective structures on a strongly tame orbifold $\orb$ with end conditions with parts of character varieties
of $\pi_{1}(\orb)$ with corresponding conditions on holonomy groups of ends.   
We mention that the uniqueness condition below simplifies the theory greatly. Otherwise, we need to 
use the sections picking the vertices and the totally geodesic planes fixed by the holonomy group of 
each end.

\subsection{ The semi-algebraic properties of $\rep^s(\pi_1(\mathcal{O}), \PGL(n+1, \bR))$ and related spaces }\label{sub-semialg}

We will now recall Section \ref{subsub-char} and make it more precise. 

A \hypertarget{term-pgroup}{{\em parabolic subalgebra}} $\mathfrak{p}$ is an algebra in a semisimple Lie algebra $\mathfrak{g}$
whose complexification contains a maximal solvable subalgebra of $\mathfrak{g}$  (p. 279--288 of \cite{Var}).
A {\em parabolic subgroup} $P$ of a semisimple Lie group $G$ is the full normalizer of a parabolic subalgebra. \index{parabolic subgroup} 

We recall from Section \ref{subsub-char}. 
Since $\mathcal{O}$ is the interior of a compact orbifold, 
there exists a finite set of generators $g_1, \dots, g_m$ with finitely many relators. 
First, $\Hom(\pi_1(\mathcal{O}), \PGL(n+1,\bR))$ can be identified with a semi-algebraic subset of 
$\PGL(n+1,\bR)^m$ corresponding to the relators. 
Each end of $\orb$ is assigned to be an \hyperlink{term-cRend}{$\cR$-type end} 
or a \hyperlink{term-cTend}{$\cT$-type end}. 

Let $\Hom_{\mathcal E}(\pi_1(\mathcal{O}), \PGL(n+1,\bR))$ denote the subspace of 
\[\Hom(\pi_1(\mathcal{O}), \PGL(n+1,\bR))\]
where the holonomy of each p-end fundamental group fixes a point of $\rpnn^n$
for an end of type $\mathcal{R}$ or 
acts on a subspace $P$ of codimension-one and on a lens meeting $P$ 
satisfying the lens-condition or a horoball tangent to $P$ for an end of type $\mathcal{T}$. 
Since there are only finitely many p-end fundamental groups up to conjugation by elements of 
$\pi_{1}(\orb)$, we obtain that 
\[\Hom_{\mathcal E}(\pi_1(\mathcal{O}), \PGL(n+1,\bR))\] is a closed semi-algebraic subset
provided that is no $\mathcal{T}$-end. 
If there are $\mathcal{T}$-ends, then we obtain a union of open subsets of closed semi-algebraic subsets.


Since  each end fundamental group is finitely generated, 
the conditions of having a common $1$-dimensional eigenspace 
for each of a finite collection of finitely generated subgroups is a semi-algebraic condition. 


Let $\rho \in \Hom_{\mathcal E}(\pi_1(E), \PGL(n+1,\bR))$
where $E$ is a \hyperlink{term-horo}{horospherical end}. 
Then $\rho(\pi_1(E))$ is virtually abelian by Theorem 1.1 of \cite{endclass}. 
Define  
\[\Hom_{\mathcal E, \mathbbm{par}} (\pi_1(E), \PGL(n+1, \bR))\] 
to be the space of representations where an abelian group
of finite index goes into a parabolic subgroup in a copy of $\PO(n, 1)$. 
By Lemma \ref{lem-parab}, 
\[\Hom_{\mathcal E, \mathbbm{par}} (\pi_1(E), \PGL(n+1, \bR))\] 
is a closed semi-algebraic set. 

\begin{lemma} \label{lem-parab}
Let $G$ be a finite extension of a finitely generated free abelian group $\bZ^m$.
Then $\Hom_{\mathcal E, \mathbbm{par}} (G, \PGL(n+1, \bR))$ is a closed algebraic set. 
\end{lemma}
\begin{proof}
Let $P$ be a maximal parabolic subgroup of a copy of $\PO(n+1, \bR)$ that fixes a point $x$. 
Then $\Hom(\bZ^m, P)$ is a closed semi-algebraic set.
\[\Hom_{\mathcal E, \mathbbm{par}}(\bZ^m, \PGL(n+1, \bR))\] equals a union 
\[\bigcup_{g\in  \PGL(n+1, \bR)} \Hom(\bZ^m, gPg^{-1}),\]
another closed semi-algebraic set.  
Now $\Hom_{\mathcal E, \mathbbm{par}}(G, \PGL(n+1, \bR))$ is a closed semi-algebraic subset of 
\[\Hom_{\mathcal E, \mathbbm{par}}(\bZ^m, \PGL(n+1, \bR)).\]
 \end{proof} 



  
  
  Let $E$ be an end orbifold of $\orb$. 
 Given \[\rho \in \Hom_{\mathcal E}(\pi_1(E), \PGL(n+1, \bR)),\] we define the following sets: 
\begin{itemize} 
\item 
Let $E$ be an end of type $\cR$.
Let \[\Hom_{\mathcal E, \mathbbm{RL}}(\pi_{1}(E),  \PGL(n+1,\bR))\] denote 
the space of  representations $h$ of $\pi_1(E)$
where $h(\pi_{1}(E))$ acts on a lens-cone $\{p\} \ast L $ for a lens $L$ and $p\not\in \clo(L)$ of a p-end $\tilde E$ 
corresponding to $E$ and the lens $L$ itself. 
Thus, it is an open subspace of the above semi-algebraic set
$\Hom_{\mathcal E}(\pi_1(E), \PGL(n+1, \bR))$.
\item 
Let $E$ denote an end of type $\cT$. 
Let \[\Hom_{\mathcal E, \mathbbm{TL}}(\pi_{1}(E),  \PGL(n+1,\bR))\] denote 
the space of totally geodesic representations $h$ of $\pi_1(E)$ satisfying 
the following condition: 
\begin{itemize}
\item $h(\pi_{1}(E))$ acts on an lens $L$ and a hyperspace $P$ where
\item $L^{o} \cap P \ne \emp$ and 
\item $L/h(\pi_{1}(E))$ is a compact orbifold with two strictly convex boundary components.
\end{itemize}  
\[\Hom_{\mathcal E, \mathbbm{TL}}(\pi_{1}(E),  \PGL(n+1,\bR))\] 
again an open subset of the semi-algebraic set 
\[\Hom_{\mathcal{E}}(\pi_{1}(E),  \PGL(n+1,\bR)).\] 
(This follows by the proof of Theorem 8.1 of \cite{endclass}.) 
\end{itemize} 

 Let 
\[ R_{E}: \Hom(\pi_1(\mathcal{O}), \PGL(n+1, \bR)) \ni h \ra h|\pi_{1}(E) \in \Hom(\pi_1(E), \PGL(n+1, \bR)) \] 
be the restriction map to the p-end fundamental group $\pi_1(E)$ corresponding to the end $E$ of $\mathcal{O}$. 
 
 A {\em representative set} of p-ends of $\torb$ is the subset of p-ends where 
 each end of $\orb$ has a corresponding p-end and a unique corresponding p-end.
 Let $\mathcal{R}_{\orb}$ denote the representative set of p-ends of $\torb$ of type $\mathcal{R}$,  
 and let $\mathcal{T}_{\orb}$ denote the representative set of p-ends of $\torb$ of type $\mathcal{T}$. 
We define a more symmetric space:
\[\Hom_{\mathcal{E}, ce}^s(\pi_1(\mathcal{O}), \PGL(n+1,\bR))\] to be
 {\small
 \begin{align} \label{eqn:euce2}
  & \Hom^s(\pi_1(\mathcal{O}), \PGL(n+1,\bR))\, \cap   \nonumber \\
 &  \bigg(\bigcap_{E \in {\mathcal{R}_{\orb}}} R_{E}^{-1}\Big(\Hom_{\mathcal E, \mathbbm{par}}(\pi_{1}(E),  \PGL(n+1,\bR)) \cup 
  \Hom_{\mathcal E, \mathbbm{RL}}(\pi_{1}(E),  \PGL(n+1,\bR))\Big)\bigg) \cap \nonumber \\
  &  \bigg(\bigcap_{E \in {\mathcal{T}_{\orb}}} R_{E}^{-1}\Big(\Hom_{\mathcal E, \mathbbm{par}}(\pi_{1}(E),  \PGL(n+1,\bR)) \cup \Hom_{\mathcal E, \mathbbm{TL}}(\pi_{1}(E),  \PGL(n+1,\bR))\Big)\bigg).
  \nonumber
\end{align} 
}

Hence, this is a union of open subsets of semi-algebraic sets in \[X:=\Hom^s_{\mathcal E}(\pi_1(\mathcal{O}), \PGL(n+1,\bR)).\]
(We don't claim that the union is open in $X$. 
These definitions allow for changes between horospherical ends to lens-type radial ones and totally geodesic ones.)


Let $\Hom^s_{\mathcal E, {\mathrm{u}}}(\pi_1(\mathcal{O}), \PGL(n+1,\bR))$ denote the subspace 
of \[\Hom^s_{\mathcal E}(\pi_1(\mathcal{O}), \PGL(n+1,\bR))\]
where each element $h$ satisfies the following properties:
\begin{itemize}
\item $h|\pi_{1}(\tilde E)$ fixes a unique point of $\rpnn^n$ corresponding to the common eigenspace of 
positive eigenvalues for lifts of elements of 
the p-end fundamental group $\pi_{1}(\tilde E)$ 
of $\cR$-type (recall Remark \ref{rem:SL}) and 
\item $h|\pi_{1}(\tilde E)$  has a common null-space $P$ of an eigen-$1$-forms 
which is unique under the condition that 
\begin{itemize}
\item $\pi_{1}(\tilde E)$ acts properly on a lens $L$ with $L \cap P$ with nonempty interior in $P$
or 
\item $H-\{p\}$ for a horosphere $H$ tangent to $P$ at $p$
\end{itemize} 
for each p-end fundamental group $\pi_{1}(\tilde E)$ of the end of $\mathcal{T}$-type.
\end{itemize}
We obtain the union of open subsets of semi-algebraic subsets since we need to consider finitely many generators of 
the fundamental groups of the ends  again by Lemma \ref{lem:ucont}. 

\begin{remark} \label{rem:redund} 
The lens condition is equivalent to the condition here.  We repeat it here to put these into set theoretical terms.
\end{remark} 

Since  \[\rep_{{\mathcal E}, \mathrm{u}}^s(\pi_1(\mathcal{O}), \PGL(n+1,\bR))\]
is the Hausdorff quotient of the above set with the conjugation $\PGL(n+1, \bR)$-action, 
this is the union of open subsets of semi-algebraic subset 
by Proposition 1.1 of \cite{JM}.


 We define \[\Hom_{\mathcal E, \mathrm{u, ce}}^s(\pi_1(\mathcal{O}), \PGL(n+1,\bR))\] to be the subset 
 {\small
 \begin{align} \label{eqn:Euce2}
  & \Hom_{\mathcal E, \mathrm{u}}^s(\pi_1(\mathcal{O}), \PGL(n+1,\bR)) \cap  \nonumber \\
  &  \bigg(\bigcap_{E \in {\mathcal{R}_{\orb}}} R_{E}^{-1}\Big(\Hom_{\mathcal E, \mathbbm{par}}\big(\pi_{1}(E),  \PGL(n+1,\bR)\big) \cup \Hom_{\mathcal E, \mathbbm{RL}}\big(\pi_{1}(E),  \PGL(n+1,\bR)\big)\Big)\bigg) \cap \nonumber \\
  &\bigg(\bigcap_{E \in {\mathcal{T}_{\orb}}} R_{E}^{-1}\Big(\Hom_{\mathcal E, \mathbbm{par}}\big(\pi_{1}(E),  \PGL(n+1,\bR)\big) \cup \Hom_{\mathcal E, \mathbbm{TL}}\big(\pi_{1}(E),  \PGL(n+1,\bR)\big)\Big)\bigg).
  \nonumber
\end{align} 
}

 Similarly to the above,  \[\rep_{{\mathcal E}, \mathrm{u, ce}}^s(\pi_1(\mathcal{O}), \PGL(n+1,\bR))\]
is a union of open subsets of strata in  
\[\rep_{{\mathcal E}, \mathrm{u}}^s(\pi_1(\mathcal{O}), \PGL(n+1,\bR)).\]


\begin{example} \label{exmp-endv2}
The uniqueness as above holds automatically for convex real projective orbifolds where a set of orbifold singularities 
contains a leaf of a radial foliation of an R-end neighborhood
as in Example \ref{exmp-endv}. 
By Theorem \ref{thm:niceend}, we obtain that these have lens-shaped radial ends only. 
\end{example}

\begin{lemma}\label{lem-niceend} 
Suppose that $\orb$ is a strongly tame real projective orbifold with radial ends. 
Suppose the end fundamental group of an end $E$ is 
\begin{itemize}
\item virtually generated by finite order elements
and 
\item is virtually abelian or is hyperbolic. 
\end{itemize} 
Suppose that the end orbifold $\Sigma_{E}$ is convex. 
Then the end $E$ is either properly convex lens-type radial end or 
is horospherical. 
\end{lemma} 

\begin{theorem} \label{thm:niceend}
Suppose that $\orb$ is a strongly tame properly convex real projective orbifold with radial ends. 
Suppose that each end fundamental group is 
\begin{itemize}
\item virtually generated by finite order elements
and 
\item is virtually abelian or is hyperbolic. 
\end{itemize} 
Then the holonomy is in 
\[\Hom_{\mathcal E, \mathrm{u, ce}}^s(\pi_{1}(\orb), \PGL(n+1, \bR)).\] 
\end{theorem} 

We need the end classification results from \cite{End1}, \cite{End2}, and \cite{End3}.

We immediately obtain: 
\begin{corollary} \label{cor:hyperbolic} 
	Let $M$ be a real projective 
	orbifold with radial ends. 
Suppose that $M$ admits finite-volume hyperbolic $3$-orbifold with ends that are \hyperlink{term-horo}{horospherical}. 
Suppose that the end orbifold is either a small orbifold with cone points of orders $3$ or a disk orbifold with corner reflectors orders $6$.
Then the ends must be of lens-type R-ends or horospherical R-ends. 
\end{corollary} 
\begin{proof} 
Each end fundamental group is virtually abelian of rank $2$. 
Hence, the end orbifold is finitely covered by a $2$-torus.
By the classification of real projective $2$-torus \cite{BenTor}, and the existence of order $3$ or $6$ singularities, 
the torus must be properly convex or complete affine. 
By the proof of Proposition 4.6 of \cite{End1}, the end is either horospherical or of lens-type. 
(Here we just need that the end orbifold be convex.)
\end{proof}
This result may be generalized to higher dimensions but we lack the formulation. 




\subsubsection{Main theorems}\label{subsub-maintheorems} 

We now state our main results:


\begin{itemize}
\item 
We define $\Def^s_{{\mathcal E}, ce}(\mathcal{O})$ to be the subspace of $\Def_{{\mathcal E}}(\mathcal{O})$ 
with real projective structures with \hyperlink{term-addg}{generalized admissible ends} and \hyperlink{term-st}{stable} irreducible holonomy homomorphisms.


\item 
We define 
$\CDef_{{\mathcal E}, \mathrm{u, ce}}(\mathcal{O})$ to be the subspace of 
$\Def_{\mathcal E, \mathrm{u}}(\mathcal{O})$ consisting of \hyperlink{term-spc}{SPC-structures} with \hyperlink{term-addg}{generalized admissible ends}.

\item We define 
$\SDef_{{\mathcal E}, \mathrm{u, ce}}(\mathcal{O})$ to be the subspace of 
$\Def_{{\mathcal E}, \mathrm{u, ce}}(\mathcal{O})$ consisting of \hyperlink{term-sspc}{strict SPC-structures} with \hyperlink{term-spc}{admissible ends}.
\end{itemize} 

We remark that these spaces are dual to the same type of the spaces but switching the $\mathcal{R}$-end with $\mathcal{T}$-ends and vice versa
by Proposition \ref{prop-duality}.

\begin{theorem}\label{thm-B} 
\hypertarget{thm-BC}{Let $\mathcal{O}$ be a noncompact strongly tame $n$-orbifold with generalized admissible ends. }
Assume $\partial \orb =\emp$. 
Suppose that $\mathcal{O}$  satisfies {\rm (IE)} and {\rm (NA).}  
Then 
the subspace  
\[\CDef_{{\mathcal E}, \mathrm{u, ce}}(\mathcal{O}) \subset \Def^s_{\mathcal E, \mathrm{u, ce}}(\mathcal{O})\]  is open.

Suppose further that every finite-index subgroup of $\pi_1(\mathcal{O})$ contains no nontrivial infinite nilpotent normal subgroup.
Then \hyperlink{term-hol}{$\hol$} maps
$\CDef_{{\mathcal E}, \mathrm{u, ce}}(\mathcal{O})$  homeomorphically to a union of components of 
 \[\rep_{{\mathcal E}, \mathrm{u, ce}}^{s}(\pi_1(\mathcal{O}), \PGL(n+1, \bR)).\]
\end{theorem}

\begin{theorem} \label{thm-C} 
Let $\mathcal{O}$ be a strict SPC noncompact strongly tame $n$-dimensional orbifold with admissible ends 
and satisfies {\rm (IE)} and {\rm (NA)}. 
Assume $\partial \orb =\emp$. 
Then
\begin{itemize}
\item $\pi_1(\mathcal{O})$ is relatively hyperbolic with respect to its end fundamental groups.
\item The subspace  $\SDef_{{\mathcal E}, \mathrm{u, ce}}(\mathcal{O}) \subset \Def^s_{\mathcal E, \mathrm{u, ce}}(\mathcal{O})$
of {\hyperlink{term-sspc}{strict SPC-structures}} with admissible ends is open. 
\end{itemize}
Suppose further that every finite-index subgroup of $\pi_1(\mathcal{O})$ contains no nontrivial infinite nilpotent normal subgroup.
Then $\hol$ maps the deformation space $\SDef_{{\mathcal E}, \mathrm{u, ce}}(\mathcal{O})$ of 
\hyperlink{term-sspc}{strict SPC-structures} on $\mathcal{O}$ with \hyperlink{term-add}{admissible ends} homeomorphically to 
a union of components of \[\rep_{{\mathcal E}, \mathrm{u, ce}}^{s}(\pi_1(\mathcal{O}), \PGL(n+1, \bR)).\]
\end{theorem}

We prove Theorems \ref{thm-B} and \ref{thm-C} by dividing into the openness result in Section \ref{sub-open} 
and the closedness result in Section \ref{sub-closed}.

\subsection{Openness} \label{sub-open}

We will show the following by proving Theorem \ref{thm-conv2}. 

\begin{theorem}\label{thm-conv} 
Let $\mathcal{O}$ be a noncompact strongly tame real projective $n$-orbifold 
and satisfies {\rm (IE)} and {\rm (NA)}. Assume $\partial \orb =\emp$. 
In $\Def^s_{\mathcal E, \mathrm{u, ce}}(\mathcal{O})$, the subspace  $\CDef_{{\mathcal E}, \mathrm{u, ce}}({\mathcal{O}})$ of 
\hyperlink{term-spc}{SPC-structures} with 
\hyperlink{term-addg}{generalized admissible ends} is open, and 
so is $\SDef_{{\mathcal E}, \mathrm{u, ce}}({\mathcal{O}})$.
\end{theorem}


We are given a properly real projective orbifold $\mathcal{O}$ with ends $E_1, \dots, E_{e_1}$ of $\cR$-type and 
$E_{e_1+1}, \dots, E_{e_1 + e_2}$ 
of $\cT$-type. Let us choose representative p-ends $\tilde E_1, \dots, \tilde E_{e_1}$ and $\tilde E_{e_1+1}, \dots, \tilde E_{e_1 + e_2}$.
Again, $e_1$ is the number of \hyperlink{term-cRend}{$\cR$-type ends},  and 
$e_2$ the number of \hyperlink{term-cTend}{$\cT$-type ends} of $\orb$.

We define a subspace of 
$\Hom_{\mathcal E}(\pi_1(\mathcal{O}), \PGL(n+1, \bR))$ to be as in Section \ref{sub-semialg}.

Let $\mathcal V$ be an open subset of 
\[\Hom^s_{\mathcal E}(\pi_1(\mathcal{O}), \PGL(n+1,\bR))\] invariant under the conjugation action of $\PGL(n+1, \bR)$
so that the following hold:
\begin{itemize}
\item one can choose a continuous section $s_{\mathcal V}^{(1)}: \mathcal V \ra (\rpn)^{e_1}$
sending a holonomy homomorphism to a common fixed point of $\Gamma_{\tilde E_i}$ for $i = 1, \dots, e_1$ and 
\item  $s_{\mathcal V}^{(1)}$ satisfies 
 \[s_{\mathcal V}^{(1)}(g h(\cdot) g^{-1})  = g \cdot s_{\mathcal V}^{(1)}(h(\cdot)) \hbox{ for } g \in \PGL(n+1, \bR).\] \index{section} 
 \end{itemize} 
 $s_{\mathcal V}^{(1)}$ is said to be a \hypertarget{term-section}{{\em fixed-point section}}.
 In these cases, we say that R-end structures are {\em determined by} $s_{\mathcal V}^{(1)}$.

Again we assume that for the open subset $\mathcal V$ of
\[\Hom^s_{\mathcal E}(\pi_1(\mathcal{O}), \PGL(n+1,\bR))\]
the following hold: 
\begin{itemize}
\item one can choose a continuous section $s_{\mathcal V}^{(2)}: \mathcal V \ra (\rpnn^{n\ast})^{e_2}$
sending a holonomy homomorphism to a common dual fixed point of $\pi_1(\tilde E_i)$ for $i = e_1+1, \dots, e_{1}+e_2$, 
\item $s_{\mathcal V}^{(2)}$ satisfies $s_{\mathcal V}^{(2)}(g h(\cdot) g^{-1}) 
 = (g^*)^{-1}\circ s_{\mathcal V}^{(2)}(h(\cdot))$ for $g \in \PGL(n+1, \bR)$, and
 \end{itemize}
 $s_{\mathcal V}^{(2)}$ is said to be a \hypertarget{term-dsection}{{\em dual fixed-point} section}.
 In this case, we say that T-end structures are {\em determined by} $s_{\mathcal V}^{(2)}$.
 
 We define $s_{\mathcal V}: \mathcal V \ra (\rpnn^n)^{e_1} \times (\rpnn^{n\ast})^{e_2}$
 as $ s_{\mathcal V}^{(1)} \times s_{\mathcal V}^{(2)}$ and call it a \hypertarget{term-fixings}{{\em fixing section}}.

Let $\mathcal V$ and $s_{\mathcal V}: {\mathcal{V}} \ra  (\rpnn^n)^{e_1} \times (\rpnn^{n \ast})^{e_2}$
 be as above.
\begin{itemize}
\item We define $\Def^s_{{\mathcal E}, s_{\mathcal V}, ce}(\mathcal{O})$ to be the subspace of 
$\Def_{{\mathcal E}, s_{\mathcal V}}(\mathcal{O})$ 
of real projective structures with \hyperlink{term-addg}{generalized admissible ends} with end structures determined by $s_{\mathcal V}$,
and stable irreducible holonomy homomorphisms in $\mathcal V$. 

\item We define 
$\CDef_{{\mathcal E}, s_{\mathcal V}, ce}(\mathcal{O})$ to be the subspace consisting of 
\hyperlink{term-spc}{SPC-structures} with \hyperlink{term-addg}{generalized admissible ends}
and holonomy homomorphisms in $\mathcal V$
 in $\Def^{s}_{{\mathcal E}, s_{\mathcal V}, ce}(\mathcal{O})$. 


\item We define 
$\SDef_{{\mathcal E}, s_{\mathcal V}, ce}(\mathcal{O})$ to be the subspace of consisting of \hyperlink{term-sspc}{strict SPC-structures} 
with admissible ends
and holonomy homomorphisms in $\mathcal V$
 in $\Def^{s}_{{\mathcal E}, s_{\mathcal V}, ce}(\mathcal{O})$. 

\end{itemize}

\begin{theorem}\label{thm-conv2} 
Let $\mathcal{O}$ be a noncompact strongly tame real projective $n$-orbifold with \hyperlink{term-addg}{generalized admissible ends}
and satisfies {\rm (IE)} and {\rm (NA)}. Assume $\partial \orb =\emp$. 
Choose an open $\PGL(n+1, \bR)$-conjugation invariant set
 \[\mathcal{V} \subset \Hom^s_{\mathcal E}(\pi_1(\orb), \PGL(n+1, \bR)),\]
and 
a \hyperlink{term-fixings}{fixing section}
$s_{\mathcal{V}}: {\mathcal{V}} \ra (\rpnn^n)^{e_1} \times (\rpnn^{n \ast})^{e_2}$.

Then $\CDef_{{\mathcal E}, s_{\mathcal V}, ce}(\mathcal{O})$ is open 
 in $\Def^s_{\mathcal E, s_{\mathcal V}, ce}(\mathcal{O})$,
 and so is $\SDef_{{\mathcal E}, s_{\mathcal V}, ce}(\mathcal{O})$.
 \end{theorem}


By Theorems \ref{thm-conv} and \ref{thm-A}, we obtain:
\begin{corollary}\label{cor-conv} 
Let $\mathcal{O}$ be a noncompact strongly tame real projective $n$-orbifold with \hyperlink{term-addg}{generalized  admissible ends} and 
satisfies {\rm (IE)} and {\rm (NA)}. Assume $\partial \orb =\emp$. 
Then \hyperlink{term-hol}{
\[\hol: \CDef_{{\mathcal E}, \mathrm{u, ce}}(\mathcal{O}) \ra \rep^s_{\mathcal E, \mathrm{u, ce}}(\pi_1(\mathcal{O}), \PGL(n+1, \bR))\]} is a local homeomorphism. 

Furthermore, if $\mathcal{O}$ has a strict SPC-structure with \hyperlink{term-add}{admissible ends}, then so is
\[\hol: \SDef_{{\mathcal E}, \mathrm{u, ce}}(\mathcal{O}) \ra \rep^s_{\mathcal E, \mathrm{u, ce}}(\pi_1(\mathcal{O}), \PGL(n+1, \bR)).\]
\end{corollary}


We just give a heuristic idea for proof of Corollary \ref{cor-conv} for 
$\SDef_{{\mathcal E}, \mathrm{u, ce}}(\mathcal{O}) $ here as in \cite{convMa}. 
We begin by taking a properly convex open cone $C$ in $\bR^{n+1}$ corresponding to $\torb$. 
Let $\orb$ be a properly convex real projective orbifold with \hyperlink{term-addg}{generalized admissible ends}.
Let $\Gamma = h(\pi_{1}(\orb))$  in $\PGL(n+1, \bR)$ be the holonomy group.
Then $\Gamma$ acts on $C$.
There is a Koszul-Vinberg function on $C$. The Hessian will be a $\Gamma$-invariant metric.
Thus, $\orb$ has an induced metric $\mu$ by using sections.
For each p-end $\tilde E$, we obtain a p-end neighborhood $U$. 
We approximate the p-end neighborhood $U$ very close to $\torb$ in the Hausdorff metric $\bdd^{H}$.
Let $C_{U}\subset \bR^{n+1}$ denote the open cone associated with $U$. 
We choose a Koszul-Vinberg Hessian metric on $C_{U}$ approximating that of $C$. 

Let $h_{t}$ be a parameter of representations in the appropriate character variety with $h_{0}=h$. 
Let $\Gamma_{t} = h_{t}(\pi_{1}(\orb))$ where $\Gamma_{0} = \Gamma$. 

By Theorem \ref{thm-A}, a convex real projective structure on $\orb$ with radial or totally geodesic ends 
has the holonomy homomorphism $h_{t}$ and developing maps $\dev_{t}: \torb \ra \rpn$. 
$\dev_{0}$ can be considered as the inclusion $\torb \ra \rpn$. 
$\dev_{t}$ may not be an inclusion in general. 
However, by the \hyperlink{term-addg}{generalized admissibility} of the ends, 
$\dev_{t}|U$ is an imbedding for $0 < t < \eps$ for sufficiently small $\eps > 0$. 
Here, a key point is that $U_{t}$ can be chosen to be properly convex when deforming. 
If $\Gamma_{t}| \pi_{1}(E)$ is not virtually abelian, it cannot have any horospherical representation. 
Thus, by the stability of the lens condition, we have openness for the associated end neighborhood.  
If $ \pi_{1}(E)$ is virtually abelian, $\Gamma_{0}(\pi_{1}(E))$ can be horospherical. 
If $\Gamma_{t}|\pi_{1}(E)$ is horospherical for $t> 0$, then this case is straightforward.
We consider the case when
$\Gamma_{t}|\pi_{1}(E)$, $t> 0$, becomes diagonalizable by our lens condition for holonomy homomorphisms.
By the classification of Benoist \cite{BenNil}  of projective structures with diagonalizable holonomy groups, 
we have brick decompositions for the transversal orbifold structure $\Sigma_{E, t}$ for $E$ and $\dev_{t}$. 
The transversal projective structure on $\Sigma_{E, t}$ has to have only one brick. 
Otherwise, we can find a fundamental domain of a finite cover of $\Sigma_{E, t}$ which under 
$\dev_{t}$ is noninjective. Since during the deformation the brick number doesn't change, 
we obtain a contradiction that $\dev_{t}|F_{t}$ cannot converge to a map with compact image as $t \ra 0$. 
(In other words, a sequence of real projective structures with more than one bricks 
cannot converge to a complete affine  structure.) 
Therefore, $\Sigma_{E, t}$ is properly convex for $\dev_{t}, t > 0$. 
Since the holonomy is for the lens type ones, we will have lens-type ends. 
(See Proposition 6.5 of \cite{convMa}.)
Let $C_{U_{t}} \subset \bR^{n+1}$ denote the convex open cone corresponding to $U_{t} = \dev_{t}(U) \subset \rpn$. 

The affine space $\bR^{n+1}$ is compactified as $\rpnn^{n+1}$. 
We show that $\clo(C_{U_{t}})$ changes in a continuous manner under the Hausdorff metric 
in $\rpnn^{n+1}$ containing $\bR^{n+1}$.
Also, the Hessian metric $\mu$ in the complement of the union of end neighborhoods
varies continuously to Hessian metrics $\mu_{t}$. Then we patch these metrics together to 
obtain a Hessian metric for $\orb$. 
The existence of Hessian metrics and by \hyperlink{term-addg}{generalized admissibility} of ends,
we can show the proper convexity. 
(The author learned that Cooper, Long, and Tillman \cite{CLT4} came up with the similar arguments
in slightly different settings. Also, they use different topology using developing maps of ends neighborhood. 
This makes thing simpler and maybe more clear.) 


\subsection{The closedness of convex real projective structures}\label{sub-closed}

We recall  \[\rep_{\mathcal E}^s(\pi_1(\mathcal{O}), \PGL(n+1, \bR))\] the subspace of stable irreducible characters
of \[\rep_{\mathcal E}(\pi_1(\mathcal{O}), \PGL(n+1, \bR))\]
which is shown to be the union of open subsets of semi-algebraic subsets in Section \ref{sub-semialg}, 
and denote by $\rep_{{\mathcal E}, \mathrm{u, ce}}^s(\pi_1(\mathcal{O}), \PGL(n+1, \bR))$ the subspace of stable irreducible characters
of $\rep_{{\mathcal E}, \mathrm{u, ce}}(\pi_1(\mathcal{O}), \PGL(n+1, \bR))$, an open subset of a semialgebraic set. 

In this section, we will need to discuss $\SI^n$ but only inside a proof.



\begin{theorem} \label{thm-closed1} 
Let $\mathcal{O}$ be a noncompact strongly tame SPC $n$-orbifold with 
\hyperlink{term-addg}{generalized admissible ends} and satisfies {\rm (IE)} and {\rm (NA)}. Assume $\partial \orb =\emp$, and 
that the nilpotent normal subgroups of every finite-index subgroup of $\pi_1(\mathcal{O})$ are trivial.
Then the following hold\,{\rm :} 
\begin{itemize}
\item The deformation space $\CDef_{{\mathcal E}, \mathrm{u, ce}}(\mathcal{O})$ of \hyperlink{term-spc}{SPC-structures} on $\mathcal{O}$ with generalized 
admissible ends maps under \hyperlink{term-hol}{$\hol$}
homeomorphically to a union of components of $\rep_{{\mathcal E}, \mathrm{u, ce}}^s(\pi_1(\mathcal{O}), \PGL(n+1, \bR))$.
\item 
The deformation space $\SDef_{{\mathcal E}, \mathrm{u, ce}}(\mathcal{O})$ of \hyperlink{term-sspc}{strict SPC-structures} on $\mathcal{O}$ with admissible ends maps under $\hol$
homeomorphically to the union of components of  $\rep_{{\mathcal E}, \mathrm{u, ce}}^s(\pi_1(\mathcal{O}), \PGL(n+1, \bR))$.
\end{itemize}
\end{theorem}


%

We will give some general idea behind the proofs here and show only
the closedness. 
We will use the developing maps to $\SI^{n}$.
Given a sequence $\mu_{i}$ of properly convex real projective structures on $\orb$ 
satisfying some boundary conditions, let $h_{i}$ denote the corresponding holonomy homomorphism 
with developing maps $\dev_{i}$. 
Let $K_{i} \subset \SI^{n}$ denote the closure of the images $\dev_{i}(\torb)$.
We may choose a subsequence so that $h_i \ra h$ for a 
representation $h: \pi_1(\orb) \ra \PGL(n+1, \bR)$ 
and $K_i \ra K$ for a compact convex domain $K$. 

If $K_i$ geometric converges to a convex domain $K$ with empty interior, 
then $h$ is reducible since $K$ is contain in a proper subspace. 
Hence, $h$ is not in the target character space. 

Suppose that $K_{i}$ geometrically 
converges to a convex domain $K$ with nonempty interior $K^{o}$ by choosing a subsequence. 
Suppose that $K$ is not properly convex. Then $K$ contain a pair of antipodal points. 
We take the maximal great sphere $\SI^{i}$ in $K$ for $i \geq 0$. 
The limiting holonomy group acts on $\SI^{i}$ and hence is reducible. 
Thus, $h$ is not in the target character subspace. 

Hence, $K$ is properly convex, 
and $K^{o}/\Gamma$ is a properly convex real projective orbifold $\orb'$. 
We can show that $\orb'$ is diffeomorphic to $\orb$. 
By \cite{GM}, we can show that $h_{i}$ converges to a faithful representation $h: \pi_{1}(\orb) \ra \SLpm$. 

\subsection{Nicest cases} 

Theorems \ref{thm:niceend} and \ref{thm-closed1} imply the following: 

\begin{corollary} \label{cor-closed2}
Let $\mathcal{O}$ be a strongly tame SPC $n$-dimensional real projective orbifold with only radial ends
and satisfies {\rm (IE)} and {\rm (NA)}. 
Suppose that each end fundamental group is generated by finite order elements
and is virtually abelian or hyperbolic. 
Assume $\partial \orb =\emp$, and 
that the nilpotent normal subgroups of every finite-index subgroup of $\pi_1(\mathcal{O})$ are trivial.
Then 
\hyperlink{term-hol}{$\hol$} 
maps the deformation space $\CDef_{{\mathcal E}}(\mathcal{O})$ of \hyperlink{term-spc}{SPC-structures} on $\mathcal{O}$ homeomorphically to 
a union of components of 
\[\rep_{{\mathcal E}, \mathrm{u, ce}}^{s}(\pi_1(\mathcal{O}), \PGL(n+1, \bR)).\]  
The same can be said for $\SDef_{{\mathcal E}}(\mathcal{O})$. 
\end{corollary} 

These types of deformations from structures with cusps to ones with lens-type ends are realized in our main examples
as stated in Section \ref{sub:mainresults}.
We need the restrictions on the target space 
since  the convexity of $\orb$ is not preserved under the 
hyperbolic Dehn surgery deformations of Thurston, 
as pointed out by Cooper at ICERM in September 2013.

Strongly tame properly convex Coxeter orbifolds admitting complete hyperbolic structures will 
satisfy the premise. Also,  $2h\underbar{\,\,}1\underbar{\,\,}1$ and the double of the simplex orbifold 
do also.  

For Coxeter orbifolds, this simplifies further. 

\begin{corollary}  \label{cor-closed3}
Let $\mathcal{O}$ be a strongly tame Coxeter $n$-dimensional real projective orbifold, $n \geq 3$, with only radial ends
admitting a complete hyperbolic structure. 
Then $\SDef_{{\mathcal E}, \mathrm{u, ce}}(\mathcal{O})$ is homeomorphic to 
the union of components of 
\[\rep_{{\mathcal E}}^{s}(\pi_1(\mathcal{O}), \PGL(n+1, \bR)).\]  Finally, 
 \[\SDef_{{\mathcal E}, \mathrm{u, ce}}(\mathcal{O}) = \SDef_{{\mathcal E}}(\mathcal{O}).\]
\end{corollary}
We give a sketch of the proof. 
Consider a component $C$ of
\[ \rep_{{\mathcal E}, \mathrm{u, ce}}^{s}(\pi_1(\mathcal{O}), \PGL(n+1, \bR))\]
corresponding to a component of $\SDef_{{\mathcal E}, \mathrm{u, ce}}(\mathcal{O})$. 
Let $C'$ be the inverse image of $C$ in 
\[ \Hom_{{\mathcal E}, \mathrm{u, ce}}^{s}(\pi_1(\mathcal{O}), \PGL(n+1, \bR)).\]
Then we claim that $C'$ is open in 
\[ \Hom_{{\mathcal E}}^{s}(\pi_1(\mathcal{O}), \PGL(n+1, \bR)):\]
Let $h: \pi_{1}(\mathcal{O}) \ra \PGL(n+1,\bR)$ be a representation in $C'$. 
By Theorem \ref{thm-A}, there is a neighborhood $J$ of $h$ realized by 
orbifold $\orb_{b}$ diffeomorphic to $\orb$ for each $b \in J$. 
Each end is convex since a compact projective Coxeter $(n-1)$-orbifold, $n-1 \geq 2$, admitting 
a Euclidean structure, is always convex by Vinberg \cite{Vin}. 
By Lemma \ref{lem-niceend}, the end is properly convex of lens-type or is horospherical.
Hence, $h$ is in \[ \Hom_{{\mathcal E}, \mathrm{u, ce}}^{s}(\pi_1(\mathcal{O}), \PGL(n+1, \bR)).\]
The closedness follows as in Section \ref{sub-closed}. 


\bibliographystyle{amsplain}



\bibliography{ABCbib}{}

\providecommand{\bysame}{\leavevmode\hbox to3em{\hrulefill}\thinspace}
\providecommand{\MR}{\relax\ifhmode\unskip\space\fi MR }
\providecommand{\MRhref}[2]{%
  \href{http://www.ams.org/mathscinet-getitem?mr=#1}{#2}
}
\providecommand{\href}[2]{#2}
\begin{thebibliography}{10}

\bibitem{Ballas}
Samuel~A. Ballas, \emph{Deformations of noncompact projective manifolds},
  Algebr. Geom. Topol. \textbf{14} (2014), no.~5, 2595--2625. \MR{3276842}

\bibitem{Ballas2}
\bysame, \emph{Finite volume properly convex deformations of the figure-eight
  knot}, Geom. Dedicata \textbf{178} (2015), 49--73. \MR{3397481}

\bibitem{BCL}
Samuel~A. Ballas, Daryl Cooper, and Arielle Leitner, \emph{A classification of
  generalized cusps on properly convex projective $n$-manifolds}, in
  preparation.

\bibitem{BDL}
Samuel~A. Ballas, Jeffrey Danciger, and Gye-Seon Lee, \emph{Convex projective
  structures on non-hyperbolic three-manifolds}, arXiv:1508.04794.

\bibitem{BenNil}
Yves Benoist, \emph{Nilvari\'et\'es projectives}, Comment. Math. Helv.
  \textbf{69} (1994), no.~3, 447--473. \MR{1289337}

\bibitem{BenTor}
\bysame, \emph{Tores affines}, Crystallographic groups and their
  generalizations ({K}ortrijk, 1999), Contemp. Math., vol. 262, Amer. Math.
  Soc., Providence, RI, 2000, pp.~1--37. \MR{1796124}

\bibitem{Ben0}
\bysame, \emph{Convexes divisibles}, C. R. Acad. Sci. Paris S\'er. I Math.
  \textbf{332} (2001), no.~5, 387--390. \MR{1826621}

\bibitem{Ben1}
\bysame, \emph{Convexes divisibles. {I}}, Algebraic groups and arithmetic, Tata
  Inst. Fund. Res., Mumbai, 2004, pp.~339--374. \MR{2094116}

\bibitem{Ben3}
\bysame, \emph{Convexes divisibles. {III}}, Ann. Sci. \'Ecole Norm. Sup. (4)
  \textbf{38} (2005), no.~5, 793--832. \MR{2195260}

\bibitem{Ben4}
\bysame, \emph{Convexes divisibles. {IV}. {S}tructure du bord en dimension 3},
  Invent. Math. \textbf{164} (2006), no.~2, 249--278. \MR{2218481}

\bibitem{BG04}
N.~Bergeron and T.~Gelander, \emph{A note on local rigidity}, Geom. Dedicata
  \textbf{107} (2004), 111--131. \MR{2110758}

\bibitem{Bowditch}
Brian~H. Bowditch, \emph{A topological characterisation of hyperbolic groups},
  J. Amer. Math. Soc. \textbf{11} (1998), no.~3, 643--667. \MR{1602069}

\bibitem{BH}
Martin~R. Bridson and Andr{\'e} Haefliger, \emph{Metric spaces of non-positive
  curvature}, Grundlehren der Mathematischen Wissenschaften [Fundamental
  Principles of Mathematical Sciences], vol. 319, Springer-Verlag, Berlin,
  1999. \MR{1744486}

\bibitem{Canary}
R.~D. Canary, D.~B.~A. Epstein, and P.~L. Green, \emph{Notes on notes of
  {T}hurston [mr0903850]}, Fundamentals of hyperbolic geometry: selected
  expositions, London Math. Soc. Lecture Note Ser., vol. 328, Cambridge Univ.
  Press, Cambridge, 2006, With a new foreword by Canary, pp.~1--115.
  \MR{2235710}

\bibitem{endclass}
S.~Choi, \emph{A classification of radial and totally geodesic ends of properly
  convex real projective orbifolds}, arXiv:1304.1605.

\bibitem{End1}
\bysame, \emph{A classification of radial or totally geodesic ends of real
  projective orbifolds. {I}: a survey of results}, arXiv:1501.00348, to appear
  in Advanced Studies in Pure Mathematics, Math. Soc. Japan.

\bibitem{End2}
\bysame, \emph{A classification of radial or totally geodesic ends of real
  projective orbifolds. {II}: properly convex ends}, arXiv:1501.00352.

\bibitem{End3}
\bysame, \emph{A classification of radial or totally geodesic ends of real
  projective orbifolds. {III}: nonproperly convex convex ends},
  arXiv:1507.00809.

\bibitem{convMa}
\bysame, \emph{The convex real projective orbifolds with radial or totally
  geodesic ends: The closedness and openness of deformations},
  arXiv:1011.1060.v3.

\bibitem{newbook}
\bysame, \emph{$\mathbb{RP}^n$-orbifolds with ends and their deformation
  spaces}, in preparation as a monograph.

\bibitem{schoimath}
\bysame, \emph{The mathematica file to compute tetrahedral real projective
  orbifolds}, \url{http://mathsci.kaist.ac.kr/~schoi/coefficientsol8.nb}.

\bibitem{cdcr2}
\bysame, \emph{Convex decompositions of real projective surfaces. {II}:
  {A}dmissible decompositions}, J. Differential Geom. \textbf{40} (1994),
  no.~2, 239--283. \MR{1293655}

\bibitem{dgorb}
\bysame, \emph{Geometric structures on orbifolds and holonomy representations},
  Geom. Dedicata \textbf{104} (2004), 161--199. \MR{2043960}

\bibitem{poly}
\bysame, \emph{The deformation spaces of projective structures on 3-dimensional
  {C}oxeter orbifolds}, Geom. Dedicata \textbf{119} (2006), 69--90.
  \MR{2247648}

\bibitem{msj}
\bysame, \emph{Geometric structures on 2-orbifolds: exploration of discrete
  symmetry}, MSJ Memoirs, vol.~27, Mathematical Society of Japan, Tokyo, 2012.
  \MR{2962023}

\bibitem{CG}
S.~Choi and W.~Goldman, \emph{The deformation spaces of convex
  {$\Bbb{RP}^2$}-structures on 2-orbifolds}, Amer. J. Math. \textbf{127}
  (2005), no.~5, 1019--1102. \MR{2170138}

\bibitem{CGLM}
S.~Choi, R.~Greene, G.~Lee, and L.~Marquis, \emph{Projective deformations of
  hyperbolic coxeter 3-orbifolds of finite volume}, in preparation.

\bibitem{CHL}
S.~Choi, C.~Hodgson, and G.~Lee, \emph{Projective deformations of hyperbolic
  {C}oxeter 3-orbifolds}, Geom. Dedicata \textbf{159} (2012), 125--167.
  \MR{2944525}

\bibitem{CLM}
S.~Choi, G.~Lee, and L.~Marquis, \emph{Convex real projective structures on
  manifolds and orbifolds}, arXiv:1605.02548, to appear in the Handbook of
  Group Actions, (L. Ji, A. Papadopoulos, S.-T. Yau, eds.), Higher Education
  Press and International Press.

\bibitem{CLT2}
D.~Cooper, D.~Long, and M.~Thistlethwaite, \emph{Computing varieties of
  representations of hyperbolic 3-manifolds into {${\rm SL}(4,\Bbb R)$}},
  Experiment. Math. \textbf{15} (2006), no.~3, 291--305. \MR{2264468}

\bibitem{CLT1}
\bysame, \emph{Flexing closed hyperbolic manifolds}, Geom. Topol. \textbf{11}
  (2007), 2413--2440. \MR{2372851}

\bibitem{CLT4}
D.~Cooper, D.~Long, and S.~Tillmann, \emph{Deforming convex projective
  manifolds}, arXiv:1511.06206.

\bibitem{CLT3}
\bysame, \emph{On convex projective manifolds and cusps}, Adv. Math.
  \textbf{277} (2015), 181--251. \MR{3336086}

\bibitem{CM}
Micka{\"e}l Crampon and Ludovic Marquis, \emph{Finitude g\'eom\'etrique en
  g\'eom\'etrie de {H}ilbert}, Ann. Inst. Fourier (Grenoble) \textbf{64}
  (2014), no.~6, 2299--2377. \MR{3331168}

\bibitem{Harpe}
Pierre de~la Harpe, \emph{On {H}ilbert's metric for simplices}, Geometric group
  theory, {V}ol.\ 1 ({S}ussex, 1991), London Math. Soc. Lecture Note Ser., vol.
  181, Cambridge Univ. Press, Cambridge, 1993, pp.~97--119. \MR{1238518}

\bibitem{DuSa}
Cornelia Dru{\c{t}}u and Mark Sapir, \emph{Tree-graded spaces and asymptotic
  cones of groups}, Topology \textbf{44} (2005), no.~5, 959--1058, With an
  appendix by Denis Osin and Mark Sapir. \MR{2153979}

\bibitem{Farb}
B.~Farb, \emph{Relatively hyperbolic groups}, Geom. Funct. Anal. \textbf{8}
  (1998), no.~5, 810--840. \MR{1650094}

\bibitem{Goldman3}
W.~Goldman, \emph{Geometric structures on manifolds and varieties of
  representations}, Geometry of group representations ({B}oulder, {CO}, 1987),
  Contemp. Math., vol.~74, Amer. Math. Soc., Providence, RI, 1988,
  pp.~169--198. \MR{957518}

\bibitem{Gconv}
\bysame, \emph{Convex real projective structures on compact surfaces}, J.
  Differential Geom. \textbf{31} (1990), no.~3, 791--845. \MR{1053346}

\bibitem{GM}
W.~Goldman and J.~Millson, \emph{Local rigidity of discrete groups acting on
  complex hyperbolic space}, Invent. Math. \textbf{88} (1987), no.~3, 495--520.
  \MR{884798}

\bibitem{Greene}
R.~Greene, \emph{The deformation theory of discrete reflection groups and
  projective structures}, Ph.D. thesis, Ohio State University, Columbus, Ohio,
  2013.

\bibitem{Gr1}
M.~Gromov, \emph{Groups of polynomial growth and expanding maps}, Inst. Hautes
  \'Etudes Sci. Publ. Math. (1981), no.~53, 53--73. \MR{623534}

\bibitem{Gr2}
\bysame, \emph{Hyperbolic groups}, Essays in group theory, Math. Sci. Res.
  Inst. Publ., vol.~8, Springer, New York, 1987, pp.~75--263. \MR{919829}

\bibitem{heard}
Damian Heard, Craig Hodgson, Bruno Martelli, and Carlo Petronio,
  \emph{Hyperbolic graphs of small complexity}, Experiment. Math. \textbf{19}
  (2010), no.~2, 211--236. \MR{2676749}

\bibitem{JM}
Dennis Johnson and John~J. Millson, \emph{Deformation spaces associated to
  compact hyperbolic manifolds}, Discrete groups in geometry and analysis
  ({N}ew {H}aven, {C}onn., 1984), Progr. Math., vol.~67, Birkh\"auser Boston,
  Boston, MA, 1987, pp.~48--106. \MR{900823}

\bibitem{Kos}
J.-L. Koszul, \emph{D\'eformations de connexions localement plates}, Ann. Inst.
  Fourier (Grenoble) \textbf{18} (1968), no.~fasc. 1, 103--114. \MR{0239529}

\bibitem{Kuiper}
N.~H. Kuiper, \emph{On convex locally-projective spaces}, Convegno
  {I}nternazionale di {G}eometria {D}ifferenziale, {I}talia, 1953, Edizioni
  Cremonese, Roma, 1954, pp.~200--213. \MR{0063115}

\bibitem{Leitner3}
Arielle Leitner, \emph{A classification of subgroups of {$SL(4,\Bbb R)$}
  isomorphic to {$\Bbb R^3$} and generalized cusps in projective 3 manifolds},
  Topology Appl. \textbf{206} (2016), 241--254. \MR{3494445}

\bibitem{Leitner1}
\bysame, \emph{Conjugacy limits of the diagonal {C}artan subgroup in
  {$SL_3(\Bbb{R})$}}, Geom. Dedicata \textbf{180} (2016), 135--149.
  \MR{3451461}

\bibitem{Leitner2}
\bysame, \emph{Limits under conjugacy of the diagonal subgroup in
  {$SL_n(\Bbb{R})$}}, Proc. Amer. Math. Soc. \textbf{144} (2016), no.~8,
  3243--3254. \MR{3503693}

\bibitem{Lok}
W.~Lok, \emph{Deformations of locally homogeneous spaces and kleinian groups},
  Ph.D. thesis, Columbia University, New York, New York, 1984, pp. 1--178.

\bibitem{ludo}
Ludovic Marquis, \emph{Espace des modules de certains poly\`edres projectifs
  miroirs}, Geom. Dedicata \textbf{147} (2010), 47--86. \MR{2660566}

\bibitem{moer}
I.~Moerdijk and D.~A. Pronk, \emph{Orbifolds, sheaves and groupoids},
  $K$-Theory \textbf{12} (1997), no.~1, 3--21. \MR{1466622}

\bibitem{DW2}
Dave~Witte Morris, \emph{Introduction to arithmetic groups}, Deductive Press,
  [place of publication not identified], 2015. \MR{3307755}

\bibitem{Munkres}
James~R. Munkres, \emph{Topology: a first course}, Prentice-Hall, Inc.,
  Englewood Cliffs, N.J., 1975. \MR{0464128}

\bibitem{PTp}
J.~Porti and S.~Tillmann, \emph{Projective structures on a hyperbolic
  3-orbifold}, working title, in preparations, 2016.

\bibitem{Thnote}
W.~Thurston, \emph{Geometry and topology of $3$-manifolds}, mimeographed notes,
  available from \url{http://library.msri.org/books/gt3m/}, 1984.

\bibitem{Thbook}
\bysame, \emph{Three-dimensional geometry and topology. {V}ol. 1}, Princeton
  Mathematical Series, vol.~35, Princeton University Press, Princeton, NJ,
  1997, Edited by Silvio Levy. \MR{1435975}

\bibitem{Var}
V.~S. Varadarajan, \emph{Lie groups, {L}ie algebras, and their
  representations}, Graduate Texts in Mathematics, vol. 102, Springer-Verlag,
  New York, 1984, Reprint of the 1974 edition. \MR{746308}

\bibitem{Vey68}
Jacques Vey, \emph{Une notion d'hyperbolicit\'e sur les vari\'et\'es localement
  plates}, C. R. Acad. Sci. Paris S\'er. A-B \textbf{266} (1968), A622--A624.
  \MR{0236837}

\bibitem{Vin3}
{\`E}.~B. Vinberg, \emph{The theory of homogeneous convex cones}, Trudy Moskov.
  Mat. Ob\v s\v c. \textbf{12} (1963), 303--358. \MR{0158414}

\bibitem{Vin}
\bysame, \emph{Discrete linear groups that are generated by reflections}, Izv.
  Akad. Nauk SSSR Ser. Mat. \textbf{35} (1971), 1072--1112. \MR{0302779}

\bibitem{Weil}
Andr{\'e} Weil, \emph{On discrete subgroups of {L}ie groups}, Ann. of Math. (2)
  \textbf{72} (1960), 369--384. \MR{0137792}

\end{thebibliography}

\end{document}